\documentclass[a4paper,10pt]{article}

\usepackage{amsmath,amssymb,amsthm,amscd}
\usepackage[mathscr]{eucal}
\usepackage[all]{xy}
\usepackage{graphicx}

\makeatletter
    
    \@addtoreset{equation}{section}
  \makeatother

\newcommand{\plim}{\varprojlim}
\newcommand{\mcal}{\mathcal}
\newcommand{\mbf}{\mathbf}

\newcommand{\mfrak}{\mathfrak}
\newcommand{\mbb}{\mathbb}
\newcommand{\mrm}{\mathrm}
\newcommand{\vphi}{\varphi}

\newcommand{\cO}{\mathcal{O}}

\newtheorem{theorem}{Theorem}[section]
\newtheorem{corollary}[theorem]{Corollary}
\newtheorem{lemma}[theorem]{Lemma}
\newtheorem{proposition}[theorem]{Proposition}

\theoremstyle{definition}
\newtheorem{definition}[theorem]{Definition}
\newtheorem{remark}[theorem]{Remark}

\newtheorem{example}[theorem]{Example}

\newtheorem*{acknowledgments}{Acknowledgments}

\title{Torsion of algebraic groups and 
iterate extensions associated with 
Lubin-Tate formal groups} 

\author{Yoshiyasu Ozeki\footnote{
Department of Mathematics and Physics, Faculty of Science, Kanagawa University,
  2946 Tsuchiya, Hiratsuka-shi, Kanagawa 259--1293, JAPAN
\endgraf
e-mail: {\tt ozeki@kanagawa-u.ac.jp}}
}

\begin{document}
\maketitle

\begin{abstract}
We show finiteness results on torsion points of commutative algebraic groups 
over a $p$-adic field $K$ with values in various algebraic extensions $L/K$ of infinite degree.
We mainly study the following cases: (1) $L$ is an abelian extension which is a splitting field 
of a crystalline character (such as a Lubin-Tate extension). (2) $L$ is a certain iterate extension of $K$ 
associated with Lubin-Tate formal groups, which is familiar with Kummer theory.
\end{abstract}

\setcounter{tocdepth}{1}
      \tableofcontents


\section{Introduction}
Let $p$ be a prime number. 
It is known (cf. \cite{Mat}, \cite[Theorem 5.2 (a)]{CL}) that, for any abelian variety $A$ over a $p$-adic field $K$
and any finite extension $L$ of $K$, we have 
$$
A(L)\simeq \mbb{Z}_p^{\oplus [L:\mbb{Q}_p]\cdot \mrm{dim}\ A}\oplus \mrm{(a\ finite\ group)}.
$$
Thus we may say that the free part of the Mordell-Weil group $A(L)$ is well-understood 
(in contrast to the number field case). 
Furthermore, some explicit  bounds on the  size of the torsion part $A(L)_{\mrm{tor}}$ of $A(L)$
is also studied  under certain reduction hypothesis on $A$ (cf.\ \cite{Si1}, \cite{Si2} and \cite{CX}).
We are interested in the finiteness of  $A(L)_{\mrm{tor}}$ 
in the case where $L$ is an algebraic extension of $K$ {\it of infinite degree}. 
Motivated by the Mazur's question \cite{Maz} for the Mordell-Weil group over a cyclotomic field,
Imai showed in \cite{Im} that the torsion subgroup of $A(K(\mu_{p^{\infty}}))$ is 
finite if $A$ has potential good reduction,
which is  well-known  as a powerful tool in Iwasawa theory.
Here, $\mu_{p^{\infty}}$ is the set of $p$-power roots of unity.
Some  ``generalizations'' of Imai's theorem are also known. 
One of generalization is given by Kubo and Taguchi \cite{KT};  
they showed that Imai's theorem still holds 
if we replace $K(\mu_{p^{\infty}})$ with $K(K^{1/p^{\infty}})$,
where $K^{1/p^{\infty}}$ is 
the set of $p$-power roots of all elements of $K$.
Another generalization is given in \cite{Oz},
which shows that Imai's theorem still holds if we replace 
$K(\mu_{p^{\infty}})$ with the composite field of $K$ and 
various Lubin-Tate extensions over $p$-adic fields. 

In this paper, we give further discussions of \cite{Oz} and  
consider a ``Lubin-Tate theoretic'' generalization of \cite{KT}. 
We need some notation to state our main theorems.
Let $k$ be a $p$-adic subfield of $K$ with residue field $\mbb{F}_q$,
$\cO_k$ the integer ring of $k$ 
and take a uniformizer $\pi$ of $k$.
We take any $\phi=\phi(X)\in \cO_k[\![X]\!]$ with the property that 
$\phi(X)\equiv X^q\ \mrm{mod}\ \pi$ 
and  
$\phi(X)\equiv \pi X \ \mrm{mod}\ X^2$.
We set
$$
\widetilde{K}:=K(x\in \mbf{m}_{\overline{K}} \mid  
\phi^n(x)\in K\ \mbox{for some} \ n\ge 1).
$$
Here, $\overline{K}$ is a (fixed) algebraic closure of $K$,
$\mbf{m}_{\overline{K}}$ is its maximal ideal and 
$\phi^n$ is the $n$-th composite of $\phi$.
The field $\widetilde{K}$ has a geometric interpretation as follows: 
Denote by $\mbf{m}_K$ the maximal ideal of
the integer ring $\cO_K$ of $K$.   
Let $F$ be the Lubin-Tate formal group over $\cO_k$ 
associated with $\pi$.
Then $\widetilde{K}$ is the extension field of $K$ obtained by adjoining 
to $K$ all $\pi^n$-th roots of all elements of $F(\mbf{m}_K)$ for all $n$. 
If $k=\mbb{Q}_p$ and $\pi=p>2$, then we have
$\widetilde{K}=K((\cO^{\times}_K)^{1/p^{\infty}})$ where 
$(\cO^{\times}_K)^{1/p^{\infty}}$ is 
the set of $p$-power roots of all units of $\cO_K^{\times}$
(cf.\ Example \ref{basicex}). 
Thus, in some sense, the field $\tilde{K}$ is ``closely related'' to the field studied by Kubo and Taguchi.
The field $\widetilde{K}$ satisfies the following interesting properties (see Section \ref{Mainsection}).
\begin{itemize}
\item[-] $\widetilde{K}$ is independent of the choice of $\phi$;
it depends only on the choice of $K,k$ and $\pi$.
\item[-] $\widetilde{K}$ has a finite residue field and contains $k_{\pi}$.
Here, 
$k_{\pi}/k$ is the Lubin-Tate extension associated with $\pi$. 
\item[-] $\widetilde{K}$ is a non-abelian $p$-adic Lie extension of $K$ 
and the extension $\widetilde{K}/Kk_{\pi}$ is a $\mbb{Z}_p^{\oplus [K:\mbb{Q}_p]}$-extension.
The maximal abelian extension of $K$ contained in $\widetilde{K}$ is a finite extension of $Kk_{\pi}$.
\end{itemize}
For the study of the Galois group $\mrm{Gal}(\widetilde{K}/Kk_{\pi})$, 
we follow the Kummer theory arguments of 
Ribet \cite{Ri1}, \cite{Ri2} and Banaszak-Gajda-Krason \cite{BGK}.

Our interest is to study the finiteness of torsion points of abelian varieties,
more generally, commutative algebraic groups,
with values in (a finite extension of) $\widetilde{K}$.
Now, first main result is as follows. 

\begin{theorem}
\label{MMT}
Assume that the pair $(k,\pi)$ does not satisfy  the condition {\rm (W)} explained just below.
Then, 
for any finite extension $L$ of $\widetilde{K}$ and  
any abelian variety $A$ over $L$ with potential good reduction,
the torsion subgroup $A(L)_{\mrm{tor}}$ of $A(L)$ is finite.
\end{theorem}
The condition (W) appeared in the theorem is as follows.
Let  $k_G$ be the Galois closure of $k/\mbb{Q}_p$ and set
 $d_G:=[k_G:\mbb{Q}_p]$.
Fix an embedding $\overline{\mbb{Q}}\hookrightarrow \overline{\mbb{Q}}_p$.
We say that $\alpha$ is a {\it $q$-Weil integer of weight $w$} 
(resp.\ a {\it $q$-Weil number of weight $w$} )
if $\alpha$ is an  algebraic integer (resp.\ an algebraic number) 
such that 
$|\iota(\alpha)|=q^{w/2}$ for any embedding $\iota\colon \overline{\mbb{Q}} \hookrightarrow \mbb{C}$.
Then,  the condition (W) is;
\begin{itemize}
\item[{\rm (W)}] $\mrm{Nr}_{k/\mbb{Q}_p}(\pi)$
is a $q$-Weil integer of weight  $d_G/c$
for  some integer $c$ with $1\le c\le d_G$.
\end{itemize}
Theorem \ref{MMT} is a generalization and also a slight refinement  of the main theorem of \cite{Oz}; 
it shows that, if (W) with a bit stronger assumption on the weight does not hold, then  
the torsion subgroup of $A(L)$ is 
finite for any finite extension $L$ of $k_{\pi}$ and 
any abelian variety $A$ over  $L$ with potential good reduction. 
Note that Imai's theorem can be recovered by
applying the main theorem of \cite{Oz} (or Theorem \ref{MMT} above) with  $k=\mbb{Q}_p$ and $\pi=p$. 
We remark that our results  should give applications to  Iwasawa theory, 
for example, control theorems of Selmer groups for abelian varieties 
over certain $p$-adic Lie  extensions of number fields. 
In fact, arguments of  \cite[Section 6]{KT} seem to be familiar with our results. 

We immediately see that  the condition (W) is not enough 
if we hope to remove the reduction hypothesis from Theorem \ref{MMT}
for a finiteness property of torsion points.
(In fact, consider the case where $A$ is an elliptic curve with split multiplicative reduction and 
$(k,\pi)=(\mbb{Q}_p,p)$.
Then the pair $(\mbb{Q}_p,p)$ does not satisfy (W) but $A(\widetilde{K})_{\mrm{tor}}$ is infinite since  
 $\tilde{K}$ contains $k_{\pi}=\mbb{Q}_p(\mu_{p^{\infty}})$.) 
To overcome this, we consider the following additional condition.   
\begin{itemize}
\item[{\rm ($\mu$)}] $q^{-1}\mrm{Nr}_{k/\mbb{Q}_p}(\pi)$ is a root of unity.
\end{itemize}
Our second result below can be applied to not only abelian varieties with potential 
good reduction but also all commutative algebraic groups
(here, an algebraic group over a field $F$ is a group scheme of finite type over $F$).
\begin{theorem}
\label{MT}
Assume that the pair $(k,\pi)$ satisfies neither  {\rm (W)} nor  {\rm ($\mu$)}.
Then, for any finite extension $L$ of $\widetilde{K}$
and any commutative algebraic group $G$ over $L$,
the torsion subgroup $G(L)_{\mrm{tor}}$ of $G(L)$ is finite.
\end{theorem}
We show this theorem by combining Theorem \ref{MMT}, 
a structure theorem of commutative algebraic groups and 
 a non-archimedian  rigid uniformization theorem of abelian varieties
(\cite{Ra}, \cite{BL}  and \cite{BX}).

Furthermore, we show that, for given $p$-adic fields $k\subset K$, 
there are only finitely many possibilities 
of the absolute norm $\mrm{Nr}_{k/\mbb{Q}_p}(\pi)$ 
which might admit infiniteness of 
$G(\widetilde{K})_{\mrm{tor}}$ for some algebraic group $G$ over $K$. 
Moreover, we show ``uniform'' version of this phenomenon.
We denote by $f_K$ 
the extension degree of the residue field extension of $K/\mbb{Q}_p$.

\begin{theorem}
\label{MTF}
Let $f,g>0$ be positive integers.
There exists a finite set $\mcal{W}=\mcal{W}(f,g;k)$ of $q$-Weil integers 
depending only on $f,g$ and $k$ which satisfies the following property:
If $\mrm{Nr}_{k/\mbb{Q}_p}(\pi)\notin \mcal{W}$,
then  for any finite extension $K/k$ with $f_K\le f$, 
any commutative algebraic group $G$ over $K$ of dimension at most $g$
and any finite extension $L/\widetilde{K}$, it holds that 
$G(L)_{\mrm{tor}}$ is finite.
\end{theorem}

In Section 4, we also give finiteness results, 
such as Theorems \ref{MMT} and \ref{MT},  
on torsion of commutative algebraic groups with values in  
abelian extensions which are splitting fields of some crystalline characters 
(such as Lubin-Tate extensions).
The results seem to be conceptual and theoretical but 
they covers main results of \cite{Im} and \cite{Oz}.

\begin{acknowledgments}
The author would like to express his sincere gratitude to 
T. Hiranouchi and Y. Taguchi for useful discussions.  
This work is supported by JSPS KAKENHI Grant Number JP19K03433.
\end{acknowledgments}

\vspace{5mm}
\noindent
{\bf Notation :}
For any algebraic extension $F$ of $\mbb{Q}_p$, 
we denote by $G_F$  the absolute Galois group of $F$.
We also denote by $\cO_F$ and $\mbf{m}_F$ 
the ring of integers of $F$ and its maximal ideal, respectively.
For an algebraic extension $F'/F$,
we denote by $f_{F'/F}$ 
the extension degree of the residue field extension of  $F'/F$. 
We set $f_{F}:=f_{F/\mbb{Q}_p}$ to simplify notation.
We denote by $\mu_n(F)$ the set of $n$-th roots of unity in $F$, 
$\mu_n:=\mu_n(\overline{\mbb{Q}}_p)$,
$\mu_{\ell^{\infty}}(F):=\cup_{m\ge 0} \mu_{\ell^m}(F)$
for any prime number $\ell$ and 
$\mu_{\infty}(F):=\cup_{n\ge 0} \mu_{n}(F)$. 
Finally, any representation over a field in this paper is of finite dimension.


%
%

\section{Finiteness criteria of torsion of  algebraic groups}

The aim of this section is to show that 
we may reduce arguments of
finiteness of torsion points of commutative algebraic groups
to the cases 
of tori and abelian varieties with potential good reduction.
Let $K$ be a $p$-adic field and $L$ an algebraic extension of $K$.
Let $0<g\le \infty$ and let $\ell$ be any prime number (including the case $\ell =p$).
We consider the following conditions for a fixed data $(L/K,g,\ell)$:
\begin{description}
\item[{\rm ($\mu_{\ell^{\infty}}$)}]  The set $\mu_{\ell^{\infty}}(L')$ is finite for any finite extension $L'$ of $L$.
\item[{\rm (AV${}_{\ell^{\infty}}$)}] 
For  any abelian variety $A$ over $K$ with potential good reduction
 of dimension  $\le g$, 
the set  $A(L)[\ell^{\infty}]$ is finite.
\end{description}
We also consider the following conditions for a fixed data $(L/K, g)$:
\begin{description}
\item[{\rm ($\mu_{\infty}$)}] The set $\mu_{\infty}(L')$ is finite 
for any finite extension $L'$ of $L$.
\item[{\rm (AV${}_{\infty}$)}]
For  any abelian variety $A$ over $K$ with potential good reduction
 of dimension  $\le g$, 
the set  $A(L)_{\mrm{tor}}$ is finite.
\end{description}

\begin{remark}
\label{keylem}
Suppose that $L$ is a Galois extension of $K$.
We check that the condition {\rm ($\mu_{\ell^{\infty}}$)} {\rm (}resp.\ {\rm ($\mu_{\infty}$))} holds if and only if 
the set  $T(L)[\ell^{\infty}]$ {\rm (}resp.\ $T(L)_{\mrm{tor}}${\rm )} 
is finite
for any torus $T$ over $K$.

The necessity is clear and so we show the sufficiency.
Assume that the set  $T(L)[\ell^{\infty}]$ (resp.\ $T(L)_{\mrm{tor}}$) 
is finite for any torus $T$ over $K$.
Let $L'$ be a finite extension of $L$. 
Let $T_L:=\mrm{Res}_{L'/L}(\mbb{G}_{\mrm{m}})$
 be the Weil restriction of $\mbb{G}_{\mrm{m}}$.
We have $T_L(L)=\mbb{G}_{\mrm{m}}(L')$ by definition.
The torus $T_{L}$ descends to a torus $T_{K_0}$ 
over a finite subextension $K_0/K$ of $L/K$.
We set $H:=\mrm{Res}_{K_0/K}(T_{K_0})$, which is a torus over $K$. 
Since $L$ is a Galois extension of $K$, 
we have isomorphisms
$$
H(L)=T_{K_0}(L\otimes_{K} K_0)
\simeq \prod_{\sigma} T_{K_0}(L)
\simeq \prod_{\sigma} \mbb{G}_{\mrm{m}}(L')
$$
where $\sigma$ runs through all $K$-algebra embeddings $K_0\hookrightarrow L$.
Now the finiteness of  $\mu_{\ell^{\infty}}(L')$ (resp.\ $\mu_{\infty}(L')$) follows by the assumption.
\end{remark}

\begin{proposition}
\label{key}
Assume that $L$ is a Galois extension of $K$.

\noindent
{\rm (1)} If both  {\rm ($\mu_{\ell^{\infty}}$)} and {\rm (AV${}_{\ell^{\infty}}$)} hold for $(L/K,g,\ell)$,
then $G(L)[\ell^{\infty}]$ is finite 
for any commutative algebraic group $G$ over $K$ of dimension $\le g$.

\noindent
{\rm (2)} If both {\rm ($\mu_{\infty}$)} and {\rm (AV${}_{\infty}$)} hold for $(L/K,g)$,
then $G(L)_{\mrm{tor}}$ is finite
for any commutative algebraic group $G$ over $K$ of dimension $\le g$.
\end{proposition}

\begin{proof}
At first, we reduce a proof to the case where $G$ is an abelian variety.
By a structure theorem of commutative algebraic groups (cf. \cite[Theorem 2.9]{Br}),
the commutative algebraic group  $G$ lies in an 
exact sequence
$$
0\to M\times U\to G \to A\to 0
$$
of group schemes over $K$.
Here, $M$ is a subgroup scheme of a torus, $U$ is unipotent and 
$A$ is an abelian variety.
Since $U(L)$ is torsion free, 
we have exact sequences
\begin{equation}
\label{exact}
0\to M(L)[\ell^{\infty}] \to G(L)[\ell^{\infty}] \to A(L)[\ell^{\infty}]\quad 
\mbox{and}\quad  
0\to M(L)_{\mrm{tor}}\to G(L)_{\mrm{tor}}\to A(L)_{\mrm{tor}}
\end{equation}
of $G_K$-modules.
The finiteness of $M(L)[\ell^{\infty}]$ (resp.\ $M(L)_{\mrm{tor}}$) 
follows from  ($\mu_{\ell^{\infty}}$) (resp.\ ($\mu_{\infty}$)).
Thus, to show the proposition,
we may assume that $G$ is an abelian variety.

In the rest of the proof, we assume that 
$G$ is an abelian variety of dimension $\le g$ and denote $G$ by $A$. 
We denote by $g_A$ the dimension of $A$. 
We use  a non-archimedian  rigid uniformization theorem of abelian varieties
(\cite{Ra}, \cite{BL}  and \cite{BX});
there exist the following data, which is called a {\it rigid uniformization} of $A$ 
(cf.\  \cite[Definition 1.1 and Theorem 1.2]{BX}):
\begin{itemize} 
\item[(i)] $S$ is a semi-abelian variety of dimension $g_A$ fits into an exact sequence
of $K$-group schemes
$0\to T \to S \to B \to 0$ where $T$ is a torus of rank $m$ and $B$
is an abelian variety which has potential good reduction,
\item[(ii)] a closed immersion of rigid $K$-groups $N^{\rm an}\hookrightarrow S^{\rm an}$ 
where  $N$ is a group scheme which is isomorphic to $\mbb{Z}^{\oplus m}$
 after a finite base extension. Here, the subscript  ``an'' is the GAGA functor, and  
\item[(iii)] a faithfully flat morphism $S^{\rm an}\to A^{\rm an}$ of rigid $K$-groups 
with kernel $N^{\rm an}$.
\end{itemize}
It holds that  $\mrm{dim}\ B\le g_A\le g$ and 
we have exact sequences
$$
0\to N(\overline{K}) \to S(\overline{K})\to A(\overline{K})\to 0 \quad \mrm{and} \quad 
0\to T(\overline{K}) \to S(\overline{K})\to B(\overline{K})\to 0
$$
of $G_K$-modules (since ``$N\to S\to A$'' are rigid analytic morphisms, the  exactness of
the former sequence might not be well-known;  see the proof of \cite[Theorem 2.3]{CX} for this).\\

\noindent
{\bf Proof of (1).} Assume  ($\mu_{\ell^{\infty}}$) and (AV${}_{\ell^{\infty}}$) for  $(L/K,g,\ell)$.
We have exact sequences
$$
0\to V_{\ell}(S)^{G_L} \to V_{\ell}(A)^{G_L}\to  \mbb{Q}_{\ell}\otimes_{\mbb{Z}} N(\overline{K}) 
\quad {\rm and} \quad 
0\to V_{\ell}(T)^{G_L} \to V_{\ell}(S)^{G_L}\to V_{\ell}(B)^{G_L}
$$
of $G_K$-modules (here we recall that $L$ is a Galois extension of $K$), where 
 $V_{\ell}(\ast)$ stands for the rational ${\ell}$-adic Tate module. 
By ($\mu_{\ell^{\infty}}$) and  (AV${}_{\ell^{\infty}}$) for $(L/K,g,\ell)$, we know that 
$V_{\ell}(T)^{G_L} $ and $V_{\ell}(B)^{G_L}$ are zero, respectively
(here, we recall that the dimension of $B$  is at most $g$). 
This gives the fact that $V_{\ell}(S)^{G_L}$ is also zero.
Hence   we have an injection 
$V_{\ell}(A)^{G_L}\hookrightarrow  \mbb{Q}_{\ell}\otimes_{\mbb{Z}} N(\overline{K}) $ 
of $G_K$-modules. 
This shows that the  $G_K$-action on $V_{\ell}(A)^{G_L}$ factors though a finite quotient.
Hence there exists a $p$-adic subfield $K'$ of $L$  
such that $V_{\ell}(A)^{G_L}=V_{\ell}(A)^{G_{K'}}$.
Since  $V_{\ell}(A)^{G_{K'}}$ is zero by the main theorem of \cite{Mat}, 
we conclude that $A(L)[{\ell}^{\infty}]$ is finite.\\

\noindent
{\bf Proof of (2).}  Assume  ($\mu_{\infty}$) and (AV${}_{\infty}$) for  $(L/K,g)$.
It follows from (1) that $A(L)[\ell^{\infty}]$ is finite for all prime numbers $\ell$.
Hence it suffices to show that $A(L)[\ell]=0$ for almost all 
prime numbers $\ell\not=p$. 
Consider exact sequences 
$$
0\to S(L)[\ell]\to A(L)[\ell]\to N(\overline{K})/\ell N(\overline{K})
\quad {\rm and} \quad 
0\to T(L)[\ell] \to S(L)[\ell]\to B(L)[\ell]
$$
of $G_K$-modules (again we recall  that $L$ is a Galois extension of $K$).
It follows from ($\mu_{\infty}$) and (AV${}_{\infty}$)  for  $(L/K,g)$
that $T(L)[\ell] $ and $B(L)[\ell]$ are  zero 
for any $\ell$ large enough, respectively (again, we recall that the  dimension of $B$  is at most $g$),
which implies $S(L)[\ell]=0$.
For such $\ell$, we have an injection 
$A(L)[\ell]\hookrightarrow N(\overline{K})/\ell N(\overline{K})$
of $G_K$-modules.
Now we take a finite extension $K'/K$ so that
$N$ is isomorphic to a constant group $\mbb{Z}^{\oplus m}$ over $K'$.
Then the existence of 
an injection 
$A(L)[\ell]\hookrightarrow N(\overline{K})/\ell N(\overline{K})$
of $G_K$-modules  shows $A(L)[\ell]=A(L\cap K')[\ell]$, 
which must be zero for any $\ell$ large enough 
by the main theorem of \cite{Mat}.
\end{proof}

\begin{remark}
In Proposition \ref{key}, 
the condition that $L$ is a {\it Galois} extension of $K$
is necessarily. Let $\varpi$ be a uniformizer of $K$
and $\varpi_n$ a  $p^n$-th power root of $\varpi$ such that $\varpi_0=\varpi$ and
  $\varpi_{n+1}^p=\varpi_n$ for all $n\ge 0$.
We set  $L:=K(\varpi_n; n\ge 0)$, which is not a Galois extension of $K$.
We see that ($\mu_{\infty}$) is satisfied for $(L/K,g)$. 
It is a result of \cite{KT} that (AV${}_{\infty}$) also hold for $(L/K,g)$.
However, if we denote by
 $E_{\varpi}$ the Tate curve over $K$ associated with $\varpi$,
then $E_{\varpi}(L)[p^{\infty}]$ is infinite.
\end{remark}

By considering Weil restrictions, 
we obtain a slight generalization of Proposition \ref{key} in the case $g=\infty$.

\begin{corollary}
\label{keycor}
Assume that $L$ is a Galois extension of $K$.

\noindent
{\rm (1)} If both  {\rm ($\mu_{\ell^{\infty}}$)} and {\rm (AV${}_{\ell^{\infty}}$)} hold for $(L/K,g=\infty,\ell)$,
then $G(M)[\ell^{\infty}]$ is finite 
for any finite extension $M/L$ and any commutative algebraic group $G$ over $M$.

\noindent
{\rm (2)} If both {\rm ($\mu_{\infty}$)} and {\rm (AV${}_{\infty}$)} hold for $(L/K,g=\infty)$,
then $G(M)_{\mrm{tor}}$ is finite
for any finite extension $M/L$ and any commutative algebraic group $G$ over $M$.
\end{corollary}

\begin{proof}
Let $G$ be a commutative algebraic group over $M$.
Let $G_0:=\mrm{Res}_{M/L} (G)$ be the Weil restriction.
Then $G_0$ descends to a commutative algebraic group 
over a $p$-adic subfield $K_0$ in $L/K$, which we also denote by $G_0$.
Setting $H:=\mrm{Res}_{K_0/K}(G_0)$, then $H$ is a commutative algebraic group 
over $K$ and we have isomorphisms
$H(L)\simeq \prod_{\sigma} G_0(L)
\simeq \prod_{\sigma} G(M)$
where $\sigma$ runs through all $K$-algebra embeddings
$K_0 \hookrightarrow L$. 
Now the results immediately follow from Proposition \ref{key}. 
\end{proof}

\section{Locally algebraic representations and the invariant $\delta_{\chi}$}

We recall  standard properties of 
locally algebraic representations (cf.\ \cite{Se}, \cite[Appendix B]{Co}).
We also introduce an invariant $\delta_{\chi}$ for crystalline characters $\chi$.
The keys in this section are Lemmas \ref{lem:char} and \ref{root}, 
which will be often used later.

\subsection{Locally algebraic representations}
Let $k$ and $E$ be $p$-adic fields and $\chi\colon G_k\to E^{\times}$
a continuous character. We often regard $\chi$ as a character of $G^{\mrm{ab}}_k$.
Let $\mrm{Art}_k\colon k^{\times}\to G^{\mrm{ab}}_k$ be the local Artin map with arithmetic normalization. 
We define $\mbb{Q}_p$-tori $\underline{k}^{\times}$ and $\underline{E}^{\times}$ to be 
the Weil restrictions of scalars
 $\underline{k}^{\times}:=\mrm{Res}_{k/\mbb{Q}_p}(\mbb{G}_m)$ 
and $\underline{E}^{\times}:=\mrm{Res}_{E/\mbb{Q}_p}(\mbb{G}_m)$.
\begin{definition}
We say that $\chi$ is {\it locally algebraic} if there exists a (necessarily unique) 
$\mbb{Q}_p$-homomorphism $\underline{k}^{\times}\to \underline{E}^{\times}$
whose restriction to $\mbb{Q}_p$-points agrees with $\chi\circ \mrm{Art}_k$ near $1$.
\end{definition}

\begin{proposition}
\label{LA1} 
Let $E(\chi)$ be the 
$\mbb{Q}_p$-representation of $G_k$ underlying a $1$-dimensional 
$E$-vector space endowed with an $E$-linear action by $G_k$ via $\chi$.
Then,  $E(\chi)$ is Hodge-Tate if and only if 
$\chi$ is locally algebraic.
\end{proposition}
\begin{proof}
The result is a consequence of  \cite[III, A.6, Corollary]{Se}.
\end{proof}

Assume that $E(\chi)$ is Hodge-Tate. 
For any $\sigma\in \Gamma_E$, let $\chi_{\sigma E}\colon I_{\sigma E}\to E^{\times}$ 
be the restriction to the inertia $ I_{\sigma E}$ of the Lubin-Tate character associated  
with any choice of uniformizer of $\sigma E$ 
(it depends on the choice of a uniformizer of $\sigma E$,
but its restriction to the inertia subgroup does not).
Then, taking a finite extension  $k'$ of  $kE$ large enough,
we have 
$$
\chi = 
\prod_{\sigma \in \Gamma_E} \sigma^{-1} \circ \chi_{\sigma E}^{h_{\sigma}}
$$
on the inertia $I_{k'}$ for some integer $h_{\sigma}$. 
We may assume that $k'$ contains the Galois closure  $\tilde{k}$ of $kE/\mbb{Q}_p$. 
Note that $\{h_{\sigma} \mid \sigma\in \Gamma_E \}$ is the set of Hodge-Tate weights of $E(\chi)$,
that is, $C\otimes_{\mbb{Q}_p} E(\chi)\simeq \oplus_{\sigma \in \Gamma_E} C(h_{\sigma})$
where $C$ is the completion of $\overline{\mbb{Q}}_p$.
Thus 
$h=h_{E(\chi)}:=\sum_{\sigma\in \Gamma_E} h_{\sigma}$
is the sum of all Hodge-Tate weights of $E(\chi)$. 
We also set $\tilde{h}:=[\tilde{k}:E]\cdot h$, 
which is the sum of all Hodge-Tate weights of $\tilde k(\chi)$. 
We denote by $k_{\chi}$ the definition field of $\chi$
and put $K_{\chi}=Kk_{\chi}$ for any $p$-adic field $K$.

\begin{lemma}
\label{lem:char}
Assume that $\chi$ is locally algebraic 
and let the notation for $\chi$ be as above.
Let $K$ and $F$ be $p$-adic fields and 
$\psi\colon G_K\to F^{\times}$ a locally algebraic character
such that the restriction of $\psi$ to an open subgroup of $G_{K_{\chi}}$ is trivial. 
We fix a lift  $\hat{u}\colon \overline{\mbb{Q}}_p\to \overline{\mbb{Q}}_p$ 
of  each $u\in \Gamma_{\tilde{k}}$. 
If $k_{\chi}$ has a finite residue field, then we have 
$$
\left(\prod_{u\in \Gamma_{\tilde{k}}} \hat{u}\circ \psi \right)^{\tilde{h}}=
\left( \prod_{u\in \Gamma_{\tilde{k}}} \hat{u}\circ \chi \right)^{\tilde{r}}
$$
on an open subgroup of $G_{Kk}$.
Here, 
$\tilde{r}=\sum_{\tau\in \Gamma_{\tilde{k}}} r_{\tau}$ for some Hodge-Tate weight $r_{\tau}$ 
of  $F(\psi)$. 
\end{lemma}

\begin{proof}
Take a Galois extension $K'/\mbb{Q}_p$ so that all the properties below hold.
\begin{itemize}
\item[(i)] $K'$ contains $K,k'$ and the Galois closure of $F/\mbb{Q}_p$,
\item[(ii)] $\psi$ is trivial on $G_{K'_{\chi}}$ and
\item[(iii)] $\psi = \prod_{\sigma \in \Gamma_F}
\sigma^{-1} \circ \chi_{\sigma F}^{n_{\sigma}}$ on $I_{K'}$
where $\chi_{\sigma F}$ 
is the restriction to $I_{\sigma F} (\supset I_{K'})$ of the Lubin-Tate character associated  
with a uniformizer of $\sigma F$.
\end{itemize}
On the other hand, we have a decomposition $\cO_E^{\times}= \mu_{\infty}(E)\times V_E$, 
where $\mu_{\infty}(E)$
is the set of roots of unity in $E$ and  $V_E$ is an open subgroup of $\cO_E^{\times}$
which is isomorphic to $\mbb{Z}_p^{[E:\mbb{Q}_p]}$.
Replacing $K'$ by a finite extension, 
we may assume that 
\begin{itemize}
\item[(iv)] $\chi(I_{K'})$ has values in $V_E$.
\end{itemize}
Note that $\chi$ coincides with
$
\prod_{\sigma \in \Gamma_E} \sigma^{-1}\circ \chi_{\sigma E}^{h_{\sigma}}
$
on $I_{K'}$ by  (i)
and the set
$\{n_{\sigma} \mid \sigma\in \Gamma_{F} \}$ 
is the set of Hodge-Tate weights of $F(\psi)$. 
Put 
$ I^{\mrm{ab}}_{K'}=\mrm{Gal}(K'^{\mrm{ab}}/K'^{\mrm{ur}})$,\ 
$I^{\mrm{ab}}_{\tilde{k}}=\mrm{Gal}(\tilde{k}^{\mrm{ab}}/\tilde{k}^{\mrm{ur}})$,\ 
$N'=\mrm{Gal}(K'^{\mrm{ur}}K'_{\chi}/K'^{\mrm{ur}})$ and 
$\tilde{N}=\mrm{Gal}(\tilde{k}^{\mrm{ur}}\tilde{k}_{\chi}/\tilde{k}^{\mrm{ur}})$.
Since  $\psi$ restricted to  $G_{K'_{\chi}}$ is trivial, 
$\psi|_{I_{K'}}$ factors through $N'$.
We have the following commutative diagram. 
$$
\displaystyle \xymatrix{
\cO_{K'}^{\times} \ar^{\simeq}_{\mrm{Art}_{K'}}[r] \ar^{\mrm{Nr}_{K'/\tilde{k}}}[d] 
& I^{\mrm{ab}}_{K'} \ar^{\mrm{res}}[r]  \ar^{\mrm{res}}[d]
& N' \ar^{\psi}[r]   \ar^{\mrm{res}}[d] 
& K'^{\times} \ar@{_{(}->}[d]
\\
\cO_{\tilde{k}}^{\times}  \ar^{\simeq}_{\mrm{Art}_{\tilde{k}}}[r]
& I^{\mrm{ab}}_{\tilde{k}} \ar^{\mrm{res}}[r] 
& \tilde{N}
& \overline{\mbb{Q}}_p^{\times}
}$$
Here, ``Nr'', ``res'' and ``Art'' stand for the norm,  the restriction and the local Artin map
(with arithmetic normalization), respectively.
We claim that there exists a homomorphism  $\hat{\psi}\colon \tilde{N} \to \overline{\mbb{Q}}_p^{\times}$
which makes the above diagram commutative. 
It follows from the condition (iv) that  we may regard $N'$ as a closed submodule 
of $V_E$ via $\chi$. In particular, $N'$ is a sub $\mbb{Z}_p$-module of $V_E$.
By an elementary divisor theory, we may identify
$V_E$ and $N'$ with  
$\oplus^{[E:\mbb{Q}_p]}_{i=1} \mbb{Z}_p$ and  
$\oplus^{[E:\mbb{Q}_p]}_{i=1} p^{n_i}\mbb{Z}_p$  
for some $0\le n_i\le \infty$, respectively.
Thus we see that  there exists an extension 
$V_E\to  \overline{\mbb{Q}}_p^{\times}$ of 
$\psi\colon N'\to K'^{\times}$.
Moreover, it extends to  some character  $\cO_E^{\times}\to  \overline{\mbb{Q}}_p^{\times}$.
Composing this with $\chi\colon \tilde{N}\to \cO_E^{\times}$,
we obtain the desired $\hat{\psi}$.

We regard $\psi$ as a character of $\cO_{K'}^{\times}$ via local class field theory.
Then, we have 
$$
\psi(x)
= \prod_{\sigma \in \Gamma_F} \sigma^{-1}\mrm{Nr}_{K'/\sigma F}(x^{-1})^{n_{\sigma}}
= \prod_{\sigma' \in \Gamma_{K'}} \sigma'^{-1}(x^{-1})^{n_{\sigma'}}
$$
for $x\in \cO_{K'}^{\times}$ since $K'/\mbb{Q}_p$ is a Galois extension.
Here, $n_{\sigma'}:=n_{\sigma'|_{F}}$
for $\sigma'\in \Gamma_{K'}$. 
Then, it follows from the existence of $\hat{\psi}$ that 
we obtain
$
 \prod_{\sigma' \in \Gamma_{K'}} \sigma'^{-1}(x^{-1})^{n_{\tau \sigma'}}
= \prod_{\sigma' \in \Gamma_{K'}} \sigma'^{-1}(x^{-1})^{n_{\sigma'}}
$
for $x\in \cO_{K'}^{\times}$ and $\tau\in \mrm{Gal}(K'/\tilde{k})$.
By \cite[Lemma 2.4]{Oz}, we have $n_{\sigma'}=n_{\rho'}$ for $\sigma',\rho'\in \Gamma_{K'}$ with 
$\sigma'|_{\tilde{k}}=\rho'|_{\tilde{k}}$. 
For any $\sigma\in \Gamma_{\tilde{k}}$, 
we define $r_{\sigma}:=n_{\sigma'}$ for  a lift $\sigma'\in \Gamma_{K'}$ of $\sigma$. 
Then we have 
$
\psi(x)
=\prod_{\sigma \in \Gamma_{\tilde k}} \sigma^{-1}\mrm{Nr}_{K'/\tilde k}(x^{-1})^{r_{\sigma}}
$
for $x\in \cO_{K'}^{\times}$.
This implies that we have 
\begin{equation}
\label{eta}
\psi=\prod_{\sigma \in \Gamma_{\tilde k}} \sigma^{-1}\circ \chi_{\tilde k}^{r_{\sigma}}
\end{equation}
on $I_{K'}$
where $\chi_{\tilde k}$ 
is the restriction to $I_{\tilde k} (\supset I_{K'})$ of the Lubin-Tate character associated  
with a uniformizer of $\tilde k$.
On the other hand, 
we have 
\begin{equation}
\label{chi}
\chi 
=\prod_{\sigma \in \Gamma_E} \sigma^{-1}\circ \chi_{\sigma E}^{h_{\sigma}}
=\prod_{\sigma \in \Gamma_E} \sigma^{-1}\circ
\left(\prod_{\tau\in \Gamma_{\tilde k}, \tau|_{\sigma E}=1} \tau\circ \chi_{\tilde k}\right)^{h_{\sigma}}
=\prod_{\sigma \in \Gamma_{\tilde k}} \sigma^{-1}\circ \chi_{\tilde k}^{\tilde{h}_{\sigma}}
\end{equation}
on $I_{K'}$ where $\tilde{h}_{\sigma}:=h_{\sigma|_E}$ for $\sigma \in \Gamma_{\tilde k}$.

Recall that  $\hat{u}\colon \overline{\mbb{Q}}_p\to \overline{\mbb{Q}}_p$ 
is a lift of $u\in \Gamma_{\tilde{k}}$
and $\tilde{h}=[\tilde{k}:E]\sum_{\sigma\in \Gamma_E} h_{\sigma}$.
We set $\tilde r:=\sum_{\sigma\in \Gamma_{\tilde k}} r_{\sigma}$. 
We remark that $r_{\sigma}$ is a Hodge-Tate weight of $F(\psi)$ by definition.
It follows from \eqref{eta}  and \eqref{chi}  that we have 
$
\prod_{u\in \Gamma_{\tilde k}} \hat{u}\circ \psi
=\left(\prod_{\sigma \in \Gamma_{\tilde k}} \sigma^{-1}\circ \chi_{\tilde k}\right)^{\tilde r}
$
and 
$\prod_{u\in \Gamma_{\tilde k}} \hat{u}\circ \chi
=\left(\prod_{\sigma \in \Gamma_{\tilde k}} \sigma^{-1}\circ \chi_{\tilde k}\right)^{\tilde h}
$
on $I_{K'}$. 
Hence we obtain 
\begin{equation}
\label{etachi}
\left(\prod_{u\in \Gamma_{\tilde{k}}} \hat{u}\circ \psi \right)^{\tilde{h}}=
\left( \prod_{u\in \Gamma_{\tilde{k}}} \hat{u}\circ \chi \right)^{\tilde{r}}
\end{equation}
on $I_{K'}$.
Since  the restriction of $\psi$ and $\chi$  
to  $G_{K'_{\chi}}$ is trivial and the residue field of $k_{\chi}$
is finite, the equality \eqref{etachi} holds 
on an open subgroup of $G_{K'}$.
\end{proof}

\begin{remark}
\label{remark}
Suppose $k=E$ is a Galois extension of $\mbb{Q}_p$ and $\chi=\chi_{\pi}\colon G_k\to k^{\times}$
is the Lubin-Tate character associated with a uniformizer $\pi$ of $k$.
The argument above shows Lemma 2.5 of \cite{Oz};
if $\psi$ is as in Lemma \ref{lem:char}, 
then we have 
$$
\psi=\prod_{\sigma \in \Gamma_k} \sigma^{-1}\circ \chi_{\pi}^{r_{\sigma}} 
$$
on an open subgroup of $G_{Kk}$
with some Hodge-Tate weight $r_{\sigma}$ of  $F(\psi)$. 
In fact, this follows immediately from \eqref{eta} and 
the assumption that 
the restriction of $\psi$ to an open subgroup of $G_{k_{\pi}}$ is trivial.
\end{remark}

\subsection{The invariant $\delta_{\chi}$}

We introduce a technical invariant for crystalline characters.
The following 
observation of Conrad plays an important role.
\begin{proposition}[{\cite[Proposition B.4]{Co}}]
\label{LA2}
Let $E(\chi)$ be the 
$\mbb{Q}_p$-representation of $G_k$ underlying a $1$-dimensional 
$E$-vector space endowed with a $E$-linear action by $G_k$ via $\chi$.

\noindent
{\rm (1)} $E(\chi)$ is crystalline if and  only if 
there exists a $\mbb{Q}_p$-homomorphism 
$\chi_{\rm alg}\colon \underline{k}^{\times}\to \underline{E}^{\times}$
such that $\chi\circ \mrm{Art}_k$ and $\chi_{\rm alg}$ {\rm (}on $\mbb{Q}_p$-points{\rm )} 
coincides on $\cO_k^{\times}$.

\noindent
{\rm (2)} Let $k_0$ be the maximal unramified subextension 
of $k/\mbb{Q}_p$ and put $f=[k_0:\mbb{Q}_p]$. 
Assume that $E(\chi)$ is crystalline and let $\chi_{\rm alg}$ be as in {\rm (1)}. 
{\rm (}Note that $E(\chi^{-1})$ is also  crystalline.{\rm )}
Then, the filtered $\phi$-module 
$D^k_{\mrm{cris}}(E(\chi^{-1}))=(B_{\mrm{cris}}\otimes_{\mbb{Q}_p} E(\chi^{-1}))^{G_k}$
over $k$ is free of rank $1$ over $k_0\otimes_{\mbb{Q}_p} E$ 
and its $k_0$-linear endomorphism $\phi^f$ 
is given by the action of the product $\chi(\mrm{Art}_k(\pi))\cdot \chi_{\rm alg}^{-1}(\pi)\in E^{\times}$.  
Here, $\pi$ is any uniformizer of $k$.
\end{proposition}

\begin{definition}
\label{delta}
If $E(\chi)$ is crystalline and $\chi_{\rm alg}$ is  as in (1) of Proposition \ref{LA2}, then 
we set 
$$
\delta_{\chi}:=\chi(\mrm{Art}_k(\pi))\cdot \chi_{\rm alg}^{-1}(\pi)\in E^{\times}.
$$ 
Here, $\pi$ is any uniformizer of $k$. 
\end{definition}
Note that $\delta_{\chi}$ is independent of the choice of $\pi$ by Proposition \ref{LA2} (1).
We also note that we have $\delta_{\chi^{-1}}=\delta_{\chi}^{-1}$ by definition.

\begin{example}
\label{exLT}
Suppose $k=E$ and $\chi$ is 
the Lubin-Tate character $\chi_{\pi}\colon G_k \to k^{\times}$ associated with 
a uniformizer $\pi$ of $k$.
Then, the $\mbb{Q}_p$-homomorphism 
$\chi_{\rm alg} \colon \underline{k}^{\times}\to \underline{k}^{\times}$
corresponding to $\chi_{\pi}$  
is given by $x\mapsto x^{-1}$. 
 Thus we have $\delta_{\chi_{\pi}}=\pi$.
\end{example}

\begin{definition}
\label{Weil2}
Let $K/\mbb{Q}_p$ be a finite extension
of residual extension degree $f_K$
and $K_0/\mbb{Q}_p$ the maximal 
unramified subextension of $K/\mbb{Q}_p$.
We denote by $\vphi_{K_0}\colon K_0 \to K_0$ 
the arithmetic Frobenius of $K_0$, that is, the (unique) lift of $p$-th power map 
on the residue field of $K_0$.

\noindent
(1)
Let $D$ be a $\vphi$-module over $K_0$, that is, 
a finite dimensional $K_0$-vector space with 
$\vphi_{K_0}$-semilinear map $\vphi \colon D\to D$.
Then $\vphi^{f_K}\colon D\to D$ is a $K_0$-linear map.
We call $\mrm{det}(T-\vphi^{f_K}\mid  D)$ 
the {\it characteristic polynomial of $D$}.

\noindent
(2)
For  a $\mbb{Q}_p$-representation $V$ of $G_K$,
we set $D^K_{\mrm{cris}}(V):=(B_{\mrm{cris}}\otimes_{\mbb{Q}_p} V)^{G_K}$,
which is a filtered $\vphi$-module over $K$. 
\end{definition}


\begin{lemma}
\label{root}
Assume that  $\chi\colon G_k\to E^{\times}$ is crystalline. 
Let $\delta_{\chi}\in E^{\times}$ be as in Definition \ref{delta}. 
Let $k'$ and $E'$ be finite extensions of $k$ and $E$, respectively. 

\noindent 
{\rm (1)} $\alpha\in \overline{\mbb{Q}}_p$ is a root of the characteristic polynomial of 
the filtered $\vphi$-module $D^{k'}_{\mrm{cris}}(E'(\chi^{-1}))$ over $k'$
if and only if  $\alpha=\tau(\delta_{\chi})^{f_{k'/k}}$ for some $\tau\in \Gamma_E$.

\noindent
{\rm (2)} Assume that $E'$ is a Galois extension of $\mbb{Q}_p$.
Let $\Gamma$ be a finite set of  $\mbb{Q}_p$-algebra homomorphisms
$\overline{\mbb{Q}}_p\to \overline{\mbb{Q}}_p$.
Put $\hat{\chi}=\prod_{\hat{u}\in \Gamma} \hat{u}\circ \chi^{n_{\hat u}}$
for some integers $n_{\hat u}$.
We regard $\hat{\chi}$ as a {\rm (}crystalline{\rm )} character from $G_k$ to $E'^{\times}$.
Then, any root of the characteristic polynomial  of 
the filtered $\vphi$-module $D^{k'}_{\mrm{cris}}(E'(\hat{\chi}^{-1}))$ over $k'$
is of the form 
$$
\left( \prod_{\tau \in \Gamma_E } \tau(\delta_{\chi})^{t_{\tau}} \right)^{f_{k'/k}}
$$
for some integers $t_{\tau}$ such that 
$\sum_{\tau\in \Gamma_E} t_{\tau} = \sum_{\hat u\in \Gamma} n_{\hat u}$.
Furthermore, we can take each $t_{\tau}$ as a non-negative integer 
if $n_{\hat u}\ge 0$ for any $\hat u\in \Gamma$.
\end{lemma}

\begin{proof}
For any crystalline character $\psi\colon G_k\to {E'}^{\times}$,
the set of  roots of the characteristic polynomial  of 
the filtered $\vphi$-module $D^{k'}_{\mrm{cris}}(E'(\psi))$ over $k'$
is the $f_{k'/k}$-th power of that of 
the filtered $\vphi$-module  $D^k_{\mrm{cris}}(E'(\psi))$ over $k$.
Hence it suffices to consider the case where $k'=k$. 

\noindent
(1) By Proposition \ref{LA2}, it suffices to show that 
$\alpha\in \overline{\mbb{Q}}_p$ is a root of the characteristic polynomial of 
the multiplication-by-$(1\otimes \delta_{\chi})$ map 
on the $k_0$-vector space $k_0\otimes_{\mbb{Q}_p} E'$ 
if and only if $\alpha$ is 
of the form $\tau(\delta_{\chi})$ for some $\tau\in \Gamma_E$.
Since this characteristic polynomial coincides with 
that of  the multiplication-by-$\delta_{\chi}$ map 
on the $\mbb{Q}_p$-vector space $E'$,
the result immediately follows (cf.\ \cite[Proposition 2.6]{Ne}). 

\noindent
(2)
For crystalline characters $\chi_1,\chi_2\colon G_k\to E'^{\times}$,
roots of the characteristic polynomial of 
$D^k_{\mrm{cris}}(E'(\chi_1\chi_2))$ is a product of those of 
$D^k_{\mrm{cris}}(E'(\chi_1))$ and $D^{k}_{\mrm{cris}}(E'(\chi_2))$ 
since we have a  surjection 
$D^k_{\mrm{cris}}(E'(\chi_1))\otimes_{k_0} D^{k}_{\mrm{cris}}(E'(\chi_2))
\to D^k_{\mrm{cris}}(E'(\chi_1\chi_2))$ induced from the natural map
$E'(\chi_1)\otimes_{\mbb{Q}_p} E'(\chi_2)\to E'(\chi_1)\otimes_{E'} E'(\chi_2)=E'(\chi_1\chi_2)$.
Therefore, it suffices to show that, for any $\mbb{Q}_p$-algebra homomorphism
$\hat{u}\colon \overline{\mbb{Q}}_p\to \overline{\mbb{Q}}_p$, 
any root of the characteristic polynomial  of 
$D^k_{\mrm{cris}}(E'(\hat u \circ \chi^{-1}))$
is of the form $\tau(\delta_{\chi})$ with some $\tau\in \Gamma_E$.
Comparing traces of two $\mbb{Q}_p$-representations 
$E'(\hat u \circ \chi^{-1})$ and  $E'(\chi^{-1})$ of $G_k$,
we see that  the semi-simplifications of them are isomorphic to each others. 
Hence we may assume that $\hat{u}=\mrm{id}$. Therefore, the result follows by (1).
\end{proof}


\section{Finiteness theorems for abelian extensions obtained by crystalline characters}

Let $k$ and $E$ be $p$-adic fields
and  $\chi\colon G_k\to E^{\times}$ a continuous character.
We assume that $\chi$ is crystalline. 
We denote by $k_{\chi}$ the definition field of $\chi$
and set $K_{\chi}:=Kk_{\chi}$ for any $p$-adic field $K$.
By definition, we have $\mrm{ker}(\chi|_{G_K})=G_{K_{\chi}}$.
The aim of this section is to give some finiteness results on torsion points $G(K_{\chi})_{\mrm{tor}}$
for commutative algebraic groups $G$ over $K$,
which are generalizations of the main theorem of \cite{Oz}.
The invariant $\delta_{\chi}\in E^{\times}$ for $\chi$ defined in the previous section
plays a crucial role for our results. 
Let $\tilde{k}$ be  the Galois closure of $kE/\mbb{Q}_p$ and set $\tilde{d}:=[\tilde{k}:\mbb{Q}_p]$.

\begin{theorem}
\label{MTab}
Let $\chi\colon G_k\to E^{\times}$ be a crystalline character.
Assume that the following two conditions hold.
\begin{itemize}
\item[{\rm (H1)}] The residue field of $k_{\chi}$ is finite.
\item[{\rm (H2)}] The sum $h_{E(\chi)}$ of all Hodge-Tate weights (with multiplicity) of the 
$\mbb{Q}_p$-representation $E(\chi)$ is not zero.
\end{itemize}
Furthermore, we consider the following conditions.
\begin{itemize}
\item[{\rm (W)'}] $\mrm{Nr}_{E/\mbb{Q}_p}(\delta_{\chi})$
is a $q_k$-Weil number of weight $h_{E(\chi)}\tilde{d}/c$
for some integer $1\le c\le \tilde{d}$. 
If $h_{E(\chi)}>0$, we furthermore have that $\mrm{Nr}_{E/\mbb{Q}_p}(\delta_{\chi})$
is an algebraic integer.
\item[{\rm ($\mu$)'}] $q_k^{-h_{E(\chi)}}\mrm{Nr}_{E/\mbb{Q}_p}(\delta_{\chi})$ is a root of unity.
\end{itemize}

\noindent
{\rm (1)} Assume that {\rm (W)'} does not hold.
Then, the torsion subgroup  $A(K_{\chi})_{\mrm{tor}}$ of $A(K_{\chi})$ is finite
for any $p$-adic field $K$ and any abelian variety $A$ over $K$
with potential good reduction.

\noindent
{\rm (2)} Assume that neither {\rm (W)'} nor {\rm ($\mu$)'} holds. 
Then, the torsion subgroup  $G(K_{\chi})_{\mrm{tor}}$ of $G(K_{\chi})$ is finite
for any $p$-adic field $K$ and any commutative algebraic group $G$ over $K$.
\end{theorem}
Applying the above theorem to the case where 
$k=E$ and $\chi$ is the Lubin-Tate character associated with a uniformizer of $k$,
we can recover the main theorem of \cite{Oz}
with a slight refinement for the weight of a Weil number.
We remark that the condition $h_{E(\chi)}>0$ appeared in (W) is harmless since 
$k_{\chi}=k_{\chi^{-1}}$ and $\delta_{\chi}^{-1}=\delta_{\chi^{-1}}$. 
If  $h_{E(\chi)}$ is zero, we have

\begin{theorem}
\label{MTab2}
Let $\chi\colon G_k\to E^{\times}$ be a crystalline character.
Assume that the following two conditions hold.
\begin{itemize}
\item[{\rm (H1)}] The residue field of $k_{\chi}$ is finite.
\item[{\rm (H2)}] The sum of all Hodge-Tate weights (with multiplicity) of the 
$\mbb{Q}_p$-representation $E(\chi)$ is zero.
\end{itemize}
Then, the torsion subgroup  $G(K_{\chi})_{\mrm{tor}}$ of $G(K_{\chi})$ is finite
for any $p$-adic field $K$ and any commutative algebraic group $G$ over $K$.
\end{theorem}

\subsection{Proof of Theorem \ref{MTab}}

We show Theorem \ref{MTab}.
Throughout this section, we assume the conditions (H1) and  (H2) in the theorem. 
We have 
$
\chi = 
\prod_{\sigma \in \Gamma_E} \sigma^{-1} \circ \chi_{\sigma E}^{h_{\sigma}}
$
on an open subgroup of the inertia $I_{k\tilde E}$ for some integer $h_{\sigma}$. 
Here, $\tilde E$ is the Galois closure of $E/\mbb{Q}_p$.
Note that $\{h_{\sigma}\mid \sigma\in \Gamma_E \}$
is the set of Hodge-Tate weights of $E(\chi)$.
Thus $h:=h_{E(\chi)}$  in the condition (H2)
is equal to $\sum_{\sigma\in \Gamma_E} h_{\sigma}$.
We put $\tilde{h}:=[\tilde{k}:E]\cdot h$. 
In the rest of this section, we fix the choice of a lift 
$\hat{u}\colon \overline{\mbb{Q}}_p\to \overline{\mbb{Q}}_p$ 
for each $u\in \Gamma_{\tilde k}$.

\begin{proof}[Proof of Theorem \ref{MTab} (1)]
Assume that  $A(K_{\chi})_{\mrm{tor}}$ is infinite. 
We may assume that $A$ has good reduction over $K$.
Since $K_{\chi}$ has a finite residue field and 
the reduction map of $A$ restricted to the prime to $p$-part is injective,
we see that the prime to $p$-part of $A(K_{\chi})_{\mrm{tor}}$ 
is finite. Thus our assumption implies that 
$A(K_{\chi})[p^{\infty}]$ is infinite. 
This is equivalent to the condition that  $V_p(A)^{G_{K_{\chi}}}$ is non-zero. 
Since the $G_K$-action  on the dual $(V_p(A)^{G_{K_{\chi}}})^{\vee}$ of $V_p(A)^{G_{K_{\chi}}}$
factors through an  abelian quotient, if we take a $p$-adic field $F$ large enough,
we know that any irreducible non-zero $G_K$-stable $F$-submodule  $V$
of $(V_p(A)^{G_{K_{\chi}}})^{\vee}\otimes_{\mbb{Q}_p} F$ is 1-dimensional. 
We take such $F$ so that it contains $E$ and it is a Galois extension of $\mbb{Q}_p$.
Let $\psi\colon G_K\to GL_F(V)\simeq F^{\times}$
be the character obtained by the $G_K$-action on $V$.  
We note that $V$ is crystalline and its Hodge-Tate weights are in $\{0,-1\}$.
Replacing $K$ by a finite extension of itself and $k$, 
by (H1) and Lemma \ref{lem:char},
we have 
$$
\left(\prod_{u\in \Gamma_{\tilde{k}}} \hat{u}\circ \psi \right)^{\tilde{h}}=
\left( \prod_{u\in \Gamma_{\tilde{k}}} \hat{u}\circ \chi \right)^{\tilde{r}}
$$
on $G_K$
for some $-\tilde d\le \tilde{r} \le 0$ 
(here, we recall that $\tilde d=[\tilde{k}:\mbb{Q}_p]$). 
We set
$\hat{\chi}:=\left(\prod_{u\in \Gamma_{\tilde{k}}} \hat{u}\circ \psi \right)^{\tilde{h}}
=\left( \prod_{u\in \Gamma_{\tilde{k}}} \hat{u}\circ \chi \right)^{\tilde{r}}$.

Let $\alpha$ be a root of the characteristic polynomial  of 
the filtered $\vphi$-module $D^K_{\mrm{cris}}(F(\hat{\chi}^{-1}))$ over $K$. 
Since $\hat{\chi}$ has two decompositions 
$\left(\prod_{u\in \Gamma_{\tilde{k}}} \hat{u}\circ \psi \right)^{\tilde{h}}$
and $\left( \prod_{u\in \Gamma_{\tilde{k}}} \hat{u}\circ \chi \right)^{\tilde{r}}$,
we can study $\alpha$ from two perspectives. 
First we focus on 
$\hat{\chi}=\left(\prod_{u\in \Gamma_{\tilde{k}}} \hat{u}\circ \psi \right)^{\tilde{h}}$. 
We consider the invariant $\delta_{\psi}\in F^{\times}$ for $\psi$.
(Note that we have $\delta_{\psi^{-1}}=\delta^{-1}_{\psi}$ by definition.)
It follows from Lemma \ref{root} (2) 
(with ``$\chi:=\psi$'', ``$E'/E:=F/F$'', ``$k'/k:=K/K$'')
that $\alpha$  is of the form 
$$
\alpha=\prod_{\tau \in \Gamma_F } \tau(\delta_{\psi})^{s_{\tau}}
$$
for some integers $s_{\tau}$ such that $\sum_{\tau\in \Gamma_F} s_{\tau} = \tilde{h}\tilde d$.
Furthermore,  Lemma \ref{root} (1) 
shows that $\tau(\delta_{\psi^{-1}})$
is a  root of the characteristic polynomial  of 
the filtered $\vphi$-module $D^K_{\mrm{cris}}(F(\psi))$ over $K$. 
Since  $F(\psi)=V$ is a subquotient representation of $V_p(A)^{\vee}\otimes_{\mbb{Q}_p} F$, 
the element $\tau(\delta_{\psi^{-1}})$ is also a root of that of the filtered $\vphi$-module 
$D^K_{\mrm{cris}}(V_p(A)^{\vee})$ over $K$.
The Weil conjecture  implies that
$\tau(\delta_{\psi^{-1}})$ is a $q_K$-Weil integer  of weight $1$.
Hence $\alpha^{-1}=\prod_{\tau \in \Gamma_F } \tau(\delta_{\psi^{-1}})^{s_{\tau}}$ is  
a $q_K$-Weil number  of weight $\tilde{h} \tilde{d}$.
Note that $\alpha^{-1}$ is in fact an algebraic integer if $\tilde h> 0$ 
since we can take each $s_{\tau}$ as a non-negative integer in this case.  
On the other hand, since 
$\hat{\chi}=\left( \prod_{u\in \Gamma_{\tilde{k}}} \hat{u}\circ \chi \right)^{\tilde{r}}$,
it follows from Lemma \ref{root} (2) that $\alpha$ is of the form 
$$
\alpha=\left( \prod_{\tau \in \Gamma_E } \tau(\delta_{\chi})^{t_{\tau}} \right)^{f_{K/k}}
$$
for some integers $t_{\tau}$ such that $\sum_{\tau\in \Gamma_E} t_{\tau} = \tilde{r}\tilde d$. 
Hence we conclude that 
$
\alpha^{-1}=\left( \prod_{\tau \in \Gamma_E } \tau(\delta_{\chi})^{-t_{\tau}} \right)^{f_{K/k}}
$
is a $q_K$-Weil number  of weight $\tilde{h} \tilde{d}$, and thus
$\alpha_0:=\prod_{\tau \in \Gamma_E } \tau(\delta_{\chi})^{-t_{\tau}}$ is a 
$q$-Weil number  of weight $\tilde{h} \tilde{d}$. 
If we denote by $\tilde{E}$ the Galois closure of $E/\mbb{Q}_p$,
then we have
$$
\mrm{Nr}_{\tilde{E}/\mbb{Q}_p}(\alpha_0)
=\mrm{Nr}_{\tilde{E}/\mbb{Q}_p}(\delta_{\chi})^{\sum_{\tau\in \Gamma_E } (-t_{\tau})}
=\mrm{Nr}_{E/\mbb{Q}_p}(\delta_{\chi})^{-\tilde{r}\tilde d[\tilde{E}:E]}.
$$
On the other hand, 
$\mrm{Nr}_{\tilde{E}/\mbb{Q}_p}(\alpha_0)$
 is a $q$-Weil number  of weight $\tilde{h} \tilde{d}\cdot [\tilde E\colon \mbb{Q}_p]$. 
Since $\tilde h=[\tilde k:E]\cdot h$ is not zero by the assumption (H2),
we obtain that $\tilde r$ is not zero and 
$\mrm{Nr}_{E/\mbb{Q}_p}(\delta_{\chi})$ 
is a $q$-Weil number  of weight 
$-\tilde r^{-1}\tilde{h} \cdot [\tilde{E}:E]^{-1}[\tilde E\colon \mbb{Q}_p]
=-\tilde r^{-1}h \tilde{d}$. 
Furthermore, $\mrm{Nr}_{E/\mbb{Q}_p}(\delta_{\chi})$  is an algebraic integer if $\tilde h> 0$
since so is $\alpha_0$. 
This contradicts the assumption (W)'.
\end{proof}

\begin{proof}[Proof of Theorem \ref{MTab} (2)]
By (1) and Corollary \ref{keycor}, 
it suffices to show that the set $\mu_{\infty}(K_{\chi})$ is finite 
for any finite extension $K$ of $k$.
Assume that  $\mu_{\infty}(K_{\chi})$ is infinite 
for some finite extension $K$ of $k$.
By the assumption (H1), we know that the prime-to-$p$ part of 
$\mu_{\infty}(K_{\chi})$ is finite. Thus the $p$-part of $\mu_{\infty}(K_{\chi})$ is infinite. 
This implies that $K_{\chi}$ contains all $p$-power roots of unity, that is,
the $p$-adic cyclotomic character $\chi_p\colon G_K\to \mbb{Q}_p^{\times}$
is trivial on $G_{K_{\chi}}$.
Applying Lemma \ref{lem:char} and replacing $K$ by a finite extension, 
we have 
$
\left(\prod_{u\in \Gamma_{\tilde{k}}} \hat{u}\circ \chi_p \right)^{\tilde{h}}=
\left( \prod_{u\in \Gamma_{\tilde{k}}} \hat{u}\circ \chi \right)^{\tilde d}
$
on $G_K$. 
This implies 
$
\chi_p^{\tilde d \tilde h}=
\left( \prod_{u\in \Gamma_{\tilde{k}}} \hat{u}\circ \chi \right)^{\tilde d}
$
on $G_K$. 
Set $\hat{\chi}:=\left( \prod_{u\in \Gamma_{\tilde{k}}} \hat{u}\circ \chi \right)^{\tilde d}$.
The characteristic polynomial  of 
the filtered $\vphi$-module $D^K_{\mrm{cris}}(F(\chi_p^{-\tilde d \tilde h}))$ over $K$
has a unique root 
$p^{\tilde d \tilde h f_{K/\mbb{Q}_p} }=q^{\tilde d \tilde h f_{K/k} }$.
As we have seen in the proof of Theorem \ref{MTab} (1), 
if $\alpha$ is a root of the characteristic polynomial  of 
the filtered $\vphi$-module $D^K_{\mrm{cris}}(F(\hat{\chi}^{-1}))$ over $K$,
we have 
$
\alpha=\left( \prod_{\tau \in \Gamma_E } \tau(\delta_{\chi})^{t_{\tau}} \right)^{f_{K/k}}
$
for some integers $t_{\tau}$ such that $\sum_{\tau\in \Gamma_E} t_{\tau} = \tilde d^2$. 
Since $\chi_p^{\tilde d \tilde h}=\hat{\chi}$ on $G_K$, we obtain 
$$
\left( \prod_{\tau \in \Gamma_E } \tau(\delta_{\chi})^{t_{\tau}} \right)^{f_{K/k}}
=q^{\tilde d \tilde h f_{K/k} }.
$$
Denote by $\tilde E$ the Galois closure of $E/\mbb{Q}_p$.
Taking $\mrm{Nr}_{\tilde E/\mbb{Q}_p}$ to both sides of the above equality,
we have 
$$
\mrm{Nr}_{E/\mbb{Q}_p}(\delta_{\chi})^{[\tilde E:E] \tilde d^2 f_{K/k}}
=q^{[\tilde E:\mbb{Q}_p]\tilde d \tilde h f_{K/k}}.
$$
However, this contradicts the assumption that ($\mu$)' does not hold.
\end{proof}

\subsection{Proof of Theorem \ref{MTab2}}

We show Theorem \ref{MTab2}.
We assume the conditions (H1) and  (H2) in the theorem. 
Replacing $E$ by a finite extension, 
we may assume that $E$ is a Galois extension of $\mbb{Q}_p$.
Also, replacing $k$ by a finite extension, we may assume 
that $k$ contains $E$, $k$ is a Galois extension of $\mbb{Q}_p$ and 
$\chi=\prod_{\sigma\in \Gamma_E} \sigma^{-1}\chi_E^{h_{\sigma}}$ on $I_k$. 
Note that we have  $\sum_{\sigma\in \Gamma_E} h_{\sigma}=0$
by (H2).
Let $k'$ be the maximal unramified extension of $k$ contained in $k_{\chi}$.
Since the residue field of $k_{\chi}$ is finite,
we know that $k'$ is a finite extension of $k$
and we have $\chi(G_{k'})=\chi(I_{k'})$.
Furthermore, we see that $k'$  is a Galois extension of $\mbb{Q}_p$
since $k'/k$ is unramified and $k/\mbb{Q}_p$ is a Galois extension. 
Hence, replacing $k$ by a finite extension again, 
we may assume $\chi(G_{k})=\chi(I_{k})$.
We regard $\chi$ as a character of $G^{\mrm{ab}}_k$.
Take a uniformizer $\pi$ of $k$.
Since we have $\chi(G_{k})=\chi(I_{k})$ and 
$\chi=\prod_{\sigma\in \Gamma_E} \sigma^{-1}\circ \chi_E^{h_{\sigma}}$ on $I_k$,
there exists an element $\gamma\in \cO_k^{\times}$ such that 
$$
\chi\circ \mrm{Art}_k(\pi)=\prod_{\sigma \in \Gamma_E} \sigma^{-1}\mrm{Nr}_{k/E}(\gamma)^{h_{\sigma}}.
$$
For any integer $n$ prime to $p$,
put $\pi_n=\pi \gamma n$.
Then $\pi_n$ is an uniformizer of $k$.
Let $\chi_{\pi_n}\colon G_k\to k^{\times}$ be the Lubin-Tate character associated with $\pi_n$
and  we set 
$$
\chi_n:=\prod_{\sigma\in \Gamma_E}
\left( \prod_{\tau\in \Gamma_k, \tau|_E=\sigma}  
\tau^{-1} \circ \chi_{\pi_n} \right)^{h_{\sigma}}\colon G_k\to k^{\times}.
$$
We regard $\chi_{\pi_n}$ and $\chi_n$ 
as characters of $G^{\mrm{ab}}_k$.
By definition, we have $\chi_{\pi_n}\circ \mrm{Art}_k(\pi)=\gamma n$.
Since $\sum_{\sigma\in \Gamma_E} h_{\sigma}=0$, 
we have 
\begin{align*}
\chi_n\circ \mrm{Art}_k(\pi) 
& = \prod_{\sigma\in \Gamma_E}
\left( \prod_{\tau\in \Gamma_k, \tau|_E=\sigma }  
\tau^{-1}(\gamma n) \right)^{h_{\sigma}}
= \prod_{\sigma\in \Gamma_E} 
\sigma^{-1} \mrm{Nr}_{k/E}(\gamma n)^{h_{\sigma}} \\
& =\prod_{\sigma\in \Gamma_E}
\sigma^{-1} \mrm{Nr}_{k/E}(\gamma)^{h_{\sigma}}
=\chi\circ \mrm{Art}_k(\pi).
\end{align*}
Furthermore, we  see 
$\prod_{\tau\in \Gamma_k, \tau|_E=\sigma}  \tau^{-1}\circ \chi_{\pi_n}
=\sigma^{-1}\circ \chi_E$  on $I_k$ and  thus we obtain $\chi_n=\chi$ on $I_k$.
We conclude that $\chi_n=\chi$ on $G_k$ by local class field theory.
In particular,  $k_{\chi}$ is a subfield of $k_{\pi_n}$ for  every integer $n$ prime to $p$.
Since 
$\mrm{Nr}_{k/\mbb{Q}_p}(\delta_{\chi_{\pi_n}})
=\mrm{Nr}_{k/\mbb{Q}_p}(\pi_n)=n^{[k:\mbb{Q}_p]}\mrm{Nr}_{k/\mbb{Q}_p}(\pi \gamma)$,
we can choose $n$ so  that 
neither (W)' nor ($\mu$)' in the statement of Theorem \ref{MTab}  for $\delta_{\chi_{\pi_n}}$
holds.
Therefore, Theorem \ref{MTab} shows that 
$k_{\chi}$ satisfies the desired property. This completes the proof.

\section{The field $\widetilde{K}$}
\label{Mainsection}

We use the same notation $k,\pi,\phi,\dots , \widetilde{K}$ as in the Introduction.
In this section, we study some basic properties 
of the field $\widetilde{K}$.
Furthermore, applying results in the previous section to  Lubin-Tate characters,
we show theorems given in the Introduction. 
The theory of Lubin-Tate formal groups plays a key role here. 
It may be helpful for the readers to refer 
\cite{Iw}, \cite{La}  and \cite{Yo}
for standard properties of Lubin-Tate formal groups.

\subsection{Formal groups and  $\widetilde{K}$}

Let $F_{\phi}=F_{\phi}(X,Y)\in \cO_k[\![X,Y]\!]$ 
be the formal $\cO_k$-module corresponding to $\phi$,
and denote by $[\cdot]_{\phi}\colon \cO_k\to \mrm{End}_{\cO_k}(F_{\phi})$
the ring homomorphism corresponding to the  $\cO_k$-action on $F_{\phi}$.
Note that  we have $[\pi]_{\phi}=\phi$.
We also note that, for any algebraic extension $L$ of $k$,  
$F_{\phi}(\mbf{m}_L)=\mbf{m}_L$ 
is equipped with an $\cO_k$-module structure
via $F_{\phi}$, that is, 
$x\oplus  y:=F_{\phi}(x,y)$ and $a.x:=[a]_{\phi}(x)$
for $x,y\in F_{\phi}(\mbf{m}_L)$ and $a\in \cO_k$.
By definition, $k_{\pi}$ is the extension field of $k$ obtained by adjoining 
all $\pi$-power torsion points of $F_{\phi}$. 
The isomorphism class of the formal  $\cO_k$-module $F_{\phi}$
depends not on $\phi$ but on $\pi$, and thus the field
$k_{\pi}$ is independent of the choice of $\phi$.
It follows from local class field theory that $k_{\pi}$ is 
a totally ramified abelian extension of $k$, 
and the composite field of $k_{\pi}$ and the maximal 
unramified extension $k^{\mrm{ur}}$ of $k$ 
coincides with the maximal abelian extension $k^{\mrm{ab}}$ of $k$. 
The set $F_{\phi}[\pi^n]_{\phi}$ of $\pi^n$-torsion points of $F_{\phi}$
is a free $\cO_k/\pi^n\cO_k$-module of rank one and 
$T_{\pi}:=\plim F_{\phi}[\pi^n]_{\phi}$ is a free $\cO_k$-module of rank one.
The Galois group $G_k$ acts on $T_{\pi}$ by the Lubin-Tate character
$\chi_{\pi}\colon G_k\to \cO_k^{\times}$. 
If we regard $\chi_{\pi}$ as a continuous 
character $k^{\times}\to k^{\times}$ 
by the local Artin map with arithmetic normalization,
then $\chi_{\pi}$ is characterized by the property that $\chi_{\pi}(\pi)=1$
and $\chi_{\pi}(u)=u^{-1}$ for any $u\in \cO_k^{\times}$.
For any $a\in \mbf{m}_K$, we denote by $K_{\phi,a}$ the 
extension filed of $K$ obtained by adjoining to $K$ all $x\in \mbf{m}_{\overline{K}}$ 
such that $\phi^n(x)=a$ for some $n\ge 1$.  
We see that the field  $\widetilde{K}$ given in the Introduction  is equal to 
the composite field of  all $K_{\phi,a}$ for all $a\in \mbf{m}_K$.

\begin{proposition}
\label{tildeK}
{\rm (1)} $K_{\phi,0}=Kk_{\pi}$. 

\noindent
{\rm (2)} $\widetilde{K}$ does not depend on the choice of $\phi$. Thus 
the field $\widetilde{K}$ is determined by $K,k$ and $\pi$. 
\end{proposition}

\begin{proof}
(1) The  result follows immediately from the equation 
$[\pi^n]_{\phi}(X)=\phi^n(X)$.

\noindent
(2) Take any $\phi'=\phi'(X)\in \cO_k[\![X]\!]$ with the property that 
$\phi'(X)\equiv X^q\ \mrm{mod}\ \pi$ 
and  
$\phi'(X)\equiv \pi X \ \mrm{mod}\ X^2$.
Let $F_{\phi'}=F_{\phi'}(X,Y)\in \cO_k[\![X,Y]\!]$ 
be the formal $\cO_k$-module corresponding to $\phi'$.
By \cite[Chapter 8, Theorem 3.1]{La}, 
there exists an isomorphism of formal $\cO_k$-modules 
$\theta\colon F_{\phi}\overset{\sim}{\rightarrow}  F_{\phi'}$.
This $\theta$ is an element of $X\cO_k[\![X]\!]$  and 
satisfies $\theta(X)\equiv X\ \mrm{mod}\ X^2$.
Note that there exists a unique $\theta^{-1}\in X\cO_k[\![X]\!]$ such that 
$\theta \circ \theta^{-1}=\theta^{-1} \circ \theta=X$ and 
$\theta^{-1}$ is an inverse of $\theta\colon F_{\phi}\overset{\sim}{\rightarrow}  F_{\phi'}$.
For the proof, it is enough to show $K_{\phi,a}=K_{\phi',\theta(a)}$ for any $a\in \mbf{m}_K$.
Take any $x\in \mbf{m}_{\overline{K}}$ such that $\phi^n(x)=a$ for some $n>0$.
Put $y=\theta(x)$. Then, $y\in  \mbf{m}_{\overline{K}}$ and we have  
$$
\theta(a)=\theta\circ \phi^n(x)=\theta\circ [\pi^n]_{\phi}(x)=[\pi^n]_{\phi'}\circ \theta(x)
=[\pi^n]_{\phi'}(y)=\phi'^{n}(y).
$$
Hence we have $y\in K_{\phi',\theta(a)}$.
Since $\theta^{-1}$ is an element of $\cO_k[\![X]\!]$, 
we have $x=\theta^{-1}(y)\in K_{\phi',\theta(a)}$. 
This shows $K_{\phi,a}\subset K_{\phi',\theta(a)}$.
The converse inclusion $K_{\phi,a}\supset K_{\phi',\theta(a)}$ follows by a similar argument.
\end{proof}

\begin{example}
\label{basicex}
Suppose $K=\mbb{Q}_p$ and $\pi=p>2$.
We check that the field $\widetilde{K}$ is just $K((\cO^{\times}_K)^{1/p^{\infty}})$. 
Choosing  $\phi(X)$ to be $(1+X)^p-1$,  
we see that  
$\widetilde{K}=K((1+\mbf{m}_K)^{1/p^{\infty}})$,
which is a subfield of 
$K((\cO^{\times}_K)^{1/p^{\infty}})$. 
Take any integer $n>0$ and any $\alpha_n$ such that $\alpha:=\alpha_n^{p^n}$
is an element of $\cO^{\times}_K$.
Since $p>2$, we have $\alpha=\zeta (1+\alpha')$ 
for some $\zeta\in \mu_{q_K-1}$
and $\alpha'\in \mbf{m}_K$.
Since $p^n$-th roots of $\zeta$ and  $1+\alpha'$ are elements of 
$\mu_{q_K-1}$ and $K_{\phi,\alpha'}$, respectively,
we see $\alpha_n \in \widetilde{K}$. 
This shows $K((\cO^{\times}_K)^{1/p^{\infty}})\subset \widetilde{K}$. 
\end{example}

We set $G:=\mrm{Gal}(\widetilde{K}/K)$ and $H:=\mrm{Gal}(\widetilde{K}/Kk_{\pi})$.
We often regard 
the Lubin-Tate character  $\chi_{\pi}$ as characters of 
$G_K,G$ and $G/H\simeq \mrm{Gal}(Kk_{\pi}/K)$.

\begin{lemma}
\label{Hlem}
Let  $\sigma \in G$. Assume that $\chi_{\pi}(\sigma)$ is a rational integer.
Then we have $\sigma \tau \sigma^{-1}=\tau^{\chi_{\pi}(\sigma)}$ for any $\tau\in H$.
\end{lemma}

\begin{proof}
For the proof, it suffices to show 
$$
\sigma \tau \sigma^{-1}x=\tau^{\chi_{\pi}(\sigma)}x
$$
for any $\tau\in H$ and $x\in \mbf{m}_{\overline{K}}$ with $\phi^n(x)\in \mbf{m}_K$ for some $n>0$.
We set $x(\rho):=\rho x\ominus  x$ for any $\rho\in G$. It is not difficult to check that 
$x(\rho)\in F_{\phi}[\pi^n]_{\phi}$, 
$x(\rho)\oplus  \rho x(\rho^{-1})=0$ and 
$[m]_{\phi} x(\tau) \oplus  x= \tau^m x$ for any $m\in \mbb{Z}$. 
Furthermore, we note that $G$ acts on $F_{\phi}[\pi^n]_{\phi}$ by $\chi_{\pi}$, 
$H$ acts  trivially on  $F_{\phi}[\pi^n]_{\phi}$ and 
$H$ is a normal subgroup of $G$.
Therefore,  we have
\begin{align*}
\sigma \tau \sigma^{-1}x
& = 
\sigma \tau (x\oplus  x(\sigma^{-1}))
= 
\sigma \tau (x\ominus  \sigma^{-1}x(\sigma)) \\
& = 
\sigma \tau x\ominus  \sigma \tau \sigma^{-1}x(\sigma) 
= 
\sigma (x\oplus  x(\tau))\ominus  x(\sigma)  \\
& = 
\sigma x(\tau)  \oplus   (\sigma x \ominus  x(\sigma) )
 =  [\chi_{\pi}(\sigma)]_{\phi} x(\tau) \oplus  x\\
& =
 \tau^{\chi_{\pi}(\sigma)} x
\end{align*}
as desired.
\end{proof}

In particular, we see that $H$ is abelian
since $\chi_{\pi}|_H$ is trivial. 
We study more precise information about $H$ in the next section.

The following two propositions are essentially shown by Kubo and Taguchi 
(cf.\ \cite[Lemmas 2.2 and 2.3]{KT}) but we include a proof for the sake of completeness.

\begin{proposition} 
\label{unipotent}
Let $E$ be a topological field and 
$\rho\colon G\to GL_E(V)$ an continuous  $E$-linear representation of $G$ of dimension $n$.

\noindent
(1) There exists a integer $m>0$ depending only on $K/k$, $\pi$ and $n$ such that 
$H^m$ acts unipotently on $V$.

\noindent
(2) There exists a finite index subgroup $H'$ of $H$ such that 
$H'$ acts unipotently on $V$.  
\end{proposition}

\begin{proof}
(1) Take any $\tau \in H$. Let $\lambda_1,\dots ,\lambda_n$ be the eigenvalues of $\rho(\tau)$.
Since $\chi_{\pi}\colon G\to \cO_k^{\times}$ has an open image,
there exists an integer $c\ge 0$ such that $1+p^c=\chi_{\pi}(\sigma)$ for some $\sigma \in G$.
Note that the choice of $c$ depends only on $K/k$ and  $\pi$.
By Lemma \ref{Hlem}, we have $\rho(\sigma \tau \sigma^{-1})=\rho(\tau)^{1+p^c}$.
This gives the equality  $\{\lambda_1,\dots ,\lambda_n\}=\{\lambda_1^{1+p^c},\dots ,\lambda_n^{1+p^c}\}$
as multisets of $n$-elements.
 In other words, the multiplication-by-$(1+p^c)$ gives a permutation on 
the multiset $\{\lambda_1,\dots ,\lambda_n\}$.
Hence, for any $1\le i\le n$, there exists an integer $1\le r\le n$ such that 
$\lambda_i^{(1+p^c)^r}=\lambda_i$.
If we denote by $m$ the least common multiple of integers 
$(1+p^c)^r-1$ for $1\le r\le n$, 
then we have $\lambda_i^m=1$ for any $i$.
This shows that $\tau^m$ acts unipotently on $V$. 
Since $m$ depends only on $K/k$, $\pi$ and $n$, we obtained the desired result.

\noindent
(2) Since $H$ is abelian, if we take a finite extension $E'$ of $E$ large enough, 
the semisimplification of the restriction of $V\otimes_E E'$ to $H$ is a direct sum of 
characters $H\to E'^{\times}$. It follows from (1) that these characters have finite images.
Thus the result follows.
\end{proof}

We denote by $\mbb{F},\mbb{F}_{\pi}$ and $\widetilde{\mbb{F}}$ the residue fields of 
$K$, $Kk_{\pi}$ and $\widetilde{K}$, respectively. 
Note that $\mbb{F}_{\pi}$ is a finite field since $k_{\pi}/k$ is totally ramified. 
Let $c=c(K/k,\pi)$ be the minimum integer $t\ge 0$ so that $1+p^t\in \chi_{\pi}(G_K)$
(such $c$ exists since $\chi_{\pi}\colon G_K\to \cO_k^{\times}$ has an open image).

\begin{proposition}
\label{finres}
The residue field $\widetilde{\mbb{F}}$ of $\widetilde{K}$ is finite. 
Moreover, the extension degree $[\widetilde{\mbb{F}}:\mbb{F}]$ is
a divisor of  $p^{c}[\mbb{F}_{\pi}:\mbb{F}]$.
\end{proposition}

\begin{proof}
Take $\sigma\in G$ such that $1+p^c=\chi_{\pi}(\sigma)$.
By Lemma \ref{Hlem}, we have 
$\tau^{p^c}=\sigma \tau \sigma^{-1}\tau^{-1}\in (G,G)$ for any $\tau\in H$.
On the other hand, the closure $\overline{(G,G)}$ 
of $(G,G)$ in $G$ is the Galois group $\mrm{Gal}(\widetilde{K}/M)$
where $M$ is the maximal abelian extension of $K$ contained in $\widetilde{K}$. 
Thus we have $H^{p^c}\subset \overline{(G,G)}=\mrm{Gal}(\widetilde{K}/M)\subset H$. 
This gives natural surjections 
$$
H/H^{p^c}\twoheadrightarrow H/\overline{(G,G)}=\mrm{Gal}(M/Kk_{\pi})
\twoheadrightarrow \mrm{Gal}(\widetilde{\mbb{F}}/\mbb{F}_{\pi}).
$$
In particular, the Galois group $\mrm{Gal}(\widetilde{\mbb{F}}/\mbb{F}_{\pi})$
is killed by $p^c$ and hence
$\mrm{Gal}(\widetilde{\mbb{F}}/\mbb{F})$
is killed by $p^cf$ where $f:=[\mbb{F}_{\pi}:\mbb{F}]$. 
Thus the surjection 
$\widehat{\mbb{Z}}\simeq \mrm{Gal}(\overline{\mbb{F}}/\mbb{F})
\twoheadrightarrow 
\mrm{Gal}(\widetilde{\mbb{F}}/\mbb{F})$ factors through
$\widehat{\mbb{Z}}/p^c f\widehat{\mbb{Z}}\simeq 
\mbb{Z}/p^cf\mbb{Z}$.  This finishes the proof.
\end{proof}

\begin{remark}
(1) If $p\not=2$ and $K=k$, then $\widetilde{K}/K$ is totally ramified. 
In fact, we have $c=0$ and $\mbb{F}_{\pi}=\mbb{F}$ in this case.

\noindent
(2)
Let $e_{K/k}$ be the ramification index of $K/k$.
Since $k_{\pi}/k$ is totally ramified, we see the inequality $[\mbb{F}_{\pi}:\mbb{F}]\le e_{K/k}$.
Hence we obtain $[\widetilde{\mbb{F}}:\mbb{F}]\le p^c e_{K/k}$. 
\end{remark}

\subsection{The Galois group of  $\widetilde{K}/Kk_{\pi}$}

The goal of this section is to show the following.

\begin{theorem}
\label{structure:Ktilde}
$\widetilde{K}$ is a 
$\mbb{Z}_p^{\oplus [K:\mbb{Q}_p]}$-extension of $Kk_{\pi}$.
\end{theorem}

Since an extension of a $p$-adic Lie group by a $p$-adic Lie group is 
again a $p$-adic Lie group (cf.\ \cite[Lemma 9.1]{GW}),
we have 
\begin{corollary}
$\widetilde{K}/K$ is a $p$-adic Lie extension.
\end{corollary}

Let $c=c(K/k,\pi)$ be the integer defined in the previous section.

\begin{corollary}
\label{finab}
The maximal abelian extension $M$ of $K$ contained in $\widetilde{K}$ is a 
finite extension of $Kk_{\pi}$. 
Moreover, the extension degree $[M:Kk_{\pi}]$ 
is a divisor of  $p^{c[K:\mbb{Q}_p]}$.
\end{corollary}

\begin{proof}
We use the same notation as the proof of Proposition \ref{finres}.
By Theorem \ref{structure:Ktilde}, we know that
$H/H^{p^c}$ is of order $p^{c[K:\mbb{Q}_p]}$.
Thus the result follows from 
 $H^{p^c}\subset \mrm{Gal}(\widetilde{K}/M)\subset H$.
\end{proof}

In the rest of this section, we give a proof of Theorem \ref{structure:Ktilde}.
Since  $\widetilde{K}$ does not depend on the choice of $\phi$
by Proposition \ref{tildeK},
for the proof of the theorem,   
we may suppose  that 
$\phi$ is a polynomial of degree $q$.
Let $a\in \mbf{m}_K$. 
For any integer $n>0$, we denote by $K_{\phi,a,n}$ the 
extension filed of $K$ obtained by adjoining to $K$ all $x\in \mbf{m}_{\overline{K}}$ 
such that $\phi^n(x)=a$. 
In other words, 
$K_{\phi,a,n}$ is the extension filed of $K$ obtained by adjoining to $K$  all
$\pi^n$-th roots of $a$ in $F_{\phi}$.
Note that $K_{\phi,a,n}$  is a Galois extension of $K$ and we have 
$
K_{\phi,a,n}=K(a_n, F_{\phi}[\pi^n]_{\phi}).
$ 
Here, $a_n$
is any $\pi^n$-th root of $a$ in $F_{\phi}$ 
(note that the set of the roots of 
$\phi^n(X)=a$
is just $a_n\oplus  F_{\phi}[\pi^n]_{\phi}$).
By definition, $K_{\phi,a}$  is equal to 
the composite field of  all $K_{\phi,a,n}$ for all $n$.

\begin{lemma}
\label{structure0}
Let $a,b\in \mbf{m}_K$.
We have $K_{\phi,a}=K_{\phi,b}$ if any one of the following hold.

\noindent
{\rm (1)} $a\ominus  b$ is torsion in $F_{\phi}$.

\noindent
{\rm (2)} $a=[\lambda]_{\phi}(b)$ for some non-zero $\lambda\in \cO_k$.
\end{lemma}

\begin{proof}
(1) 
Let $n>0$ be an integer and 
take any $x\in \mbf{m}_{\overline{K}}$ such that $[\pi^n]_{\phi}(x)=a$.
It suffices to show that $x$ is an element of $K_{\phi,b}$.
Taking any $b_n\in \mbf{m}_{\overline{K}}$ such that 
$[\pi^n]_{\phi}(b_n)=b$,
we have $[\pi^n]_{\phi}(x\ominus  b_n)=a\ominus  b$.
By the assumption, we have
$x\ominus  b_n\in F_{\phi}[\pi^m]_{\phi}$ 
for some $m$ large enough. 
This shows that  $x$ is an element of $K_{\phi,b,m}(\subset K_{\phi,b})$.

\noindent
(2) It suffices to consider the cases where 
$\lambda\in \cO_k^{\times}$ and $\lambda=\pi^m$. 

Assume that $\lambda\in \cO_k^{\times}$.
Let $n>0$ be an integer and 
take any $x\in \mbf{m}_{\overline{K}}$ such that $[\pi^n]_{\phi}(x)=a$. 
We have $[\pi^n]_{\phi}([\lambda^{-1}]_{\phi}(x))=b$ and thus 
$[\lambda^{-1}]_{\phi}(x)$ is an element of the maximal ideal of 
$K_{\phi,b,n}$. Since $[\lambda]_{\phi}(X)\in \cO_k[\![X]\!]$, we have
 $x=[\lambda]_{\phi}([\lambda^{-1}]_{\phi}(x))\in K_{\phi,b,n}\subset K_{\phi,b}$.
This shows $K_{\phi,a}\subset K_{\phi,b}$.
The converse inclusion $K_{\phi,a}\supset K_{\phi,b}$ follows by the same way.

Next we assume $\lambda=\pi^m$.
It is not difficult to check that
$x\in \mbf{m}_{\overline{K}}$ satisfies $[\pi^n]_{\phi}(x)=a$ for $n\le m$ 
(resp.\ $n>m$) if and only if 
$x\ominus  [\pi^{m-n}]_{\phi}(b) \in F_{\phi}[\pi^n]_{\phi}$
(resp.\ $x\ominus  b_{n-m} \in F_{\phi}[\pi^n]_{\phi}$). 
Here, $b_{n-m}$ is an element of $\mbf{m}_{\overline{K}}$ 
such that $[\pi^{n-m}]_{\phi}(b_{n-m})=b$.
Thus $K_{\phi,a,n}$ is equal to 
$K(F_{\phi}[\pi^n]_{\phi})$ (resp.\ $K(b_{n-m},F_{\phi}[\pi^n]_{\phi})$).
Now the result follows.
\end{proof}

Let $a_1,\dots , a_r$ be elements of $F_{\phi}(\mbf{m}_K)$. 
We define a continuous homomorphism
$$
\varphi_i \colon G_{Kk_{\pi}}\to T_{\pi} 
$$
as follows: 
Take a system $(a^{(n)}_i)_{n\ge 0}$ in 
$\mbf{m}_{\overline{K}}$
such that $a^{(0)}_i=a_i$ and $[\pi]_{\phi}(a^{(n+1)}_i)=a^{(n)}_i$ for any $n$,
and define a homomorphism 
$\varphi^{(n)}_i \colon G_{Kk_{\pi}}\to F_{\phi}[\pi^n]_{\phi}$
by $\varphi^{(n)}_i(\sigma):=\sigma(a^{(n)}_i)\ominus  a^{(n)}_i$.
We define $\varphi_i$ to be the inverse limit of $(\varphi^{(n)}_i)_{n\ge 0}$.
It is not difficult to check that each $\varphi_i^{(n)}$, and thus also $\varphi_i$,
is independent of the choice of 
$(a^{(n)}_i)_{n\ge 0}$.
By definition, the extension field of $Kk_{\pi}$ corresponding to the kernel 
of $\varphi_i$ is $K_{\phi,a_i}$.
We define a continuous homomorphism
$$
\Phi \colon G_{Kk_{\pi}}\to \oplus^r_{i=1} T_{\pi} 
$$
by $\Phi = \varphi_1\oplus \cdots \oplus \varphi_r$.
Note that the extension field $K_{\Phi}$ of $Kk_{\pi}$ corresponding to the kernel 
of $\Phi$ is $K_{\phi,a_1}\cdots K_{\phi,a_r}$,
that is,  
the composite field of 
$K_{\phi,a_1,n}\cdots K_{\phi,a_r,n}$ 
for all $n$. 
Note that we have an isomorphism 
$$
F_{\phi}(\mbf{m}_K)\simeq \cO_k^{\oplus [K:k]}
\oplus \text{(a finite $\pi$-power torsion group)}
$$
of $\cO_{k}$-modules.

\begin{lemma}
\label{structure1}
Assume that the images of 
$a_1,\dots ,a_r$ in $F_{\phi}(\mbf{m}_K)/\mbox{{\rm (tor)}}$ 
generate the free $\cO_k$-module $F_{\phi}(\mbf{m}_K)/\mbox{{\rm (tor)}}$
{\rm (}thus $r\ge [K:k]${\rm )}.
Then, we have $\widetilde{K}=K_{\Phi}$. 
\end{lemma}

\begin{proof}
Let $x\in \mbf{m}_K$. 
It suffices to show 
$K_{\phi,x}\subset K_{\Phi}$.
Take $\lambda_1,\dots ,\lambda_r\in \cO_k$
such that 
$x\ominus 
([\lambda_1]_{\phi}(a_1)\oplus \cdots \oplus [\lambda_r]_{\phi}(a_r))$ is torsion.
By Lemma \ref{structure0} (1),
we have $K_{\phi,x}=K_{\phi,y}$ where
$y=[\lambda_1]_{\phi}(a_1)\oplus \cdots \oplus [\lambda_r]_{\phi}(a_r)$.
Thus we may assume $x=y$.
For any $n>0$ and any $j\in \{1,2,\dots ,r\}$,
take $a^{(n)}_j\in \mbf{m}_{\overline{K}}$ 
such that $[\pi^n]_{\phi}(a^{(n)}_j)=a_j$.
Set 
$x^{(n)}:=[\lambda_1]_{\phi}(a^{(n)}_1) 
\oplus \cdots \oplus  [\lambda_r]_{\phi}(a^{(n)}_r)$.
Then
$[\pi^n]_{\phi}(x^{(n)})=x$ and 
 the set of $\pi^n$-th roots of $x$ in $F_{\phi}$
is $x^{(n)}\oplus  F_{\phi}[\pi^n]_{\phi}$.
Hence, we obtain
$$
K_{\phi,x,n}=K(x^{(n)}, F_{\phi}[\pi^n]_{\phi})
\subset K(a^{(n)}_1,\dots , a^{(n)}_r, F_{\phi}[\pi^n]_{\phi})
\subset K_{\phi,a_1}\cdots K_{\phi,a_r}=K_{\Phi}
$$
for any $n$. 
This shows $K_{\phi,x}\subset K_{\Phi}$ as desired.
\end{proof}

We say that $a_1,\dots ,a_r$ are {\it linearly independent over $\cO_k$}
if no non-trivial linear combination 
$\sum^r_{i=1} [\lambda_i]_{\phi} (a_i)$ with $\lambda_i\in \cO_k$ vanishes. 
We remark that, if $a_1,\dots ,a_r$ are linearly independent over $\cO_k$,
then it implies $r\le [K:k]$.

\begin{lemma}
\label{structure2-2}
Assume that $a_1,\dots ,a_r$ are linearly independent over $\cO_k$.

\noindent
{\rm (1)} If $\lambda_1\varphi_1+\cdots + \lambda_r\varphi_r=0$
for $\lambda_1,\dots , \lambda_r\in \cO_k$, 
then we have $\lambda_1=\cdots =\lambda_r=0$.

\noindent
{\rm (2)} The map $\Phi \colon G_{Kk_{\pi}}\to \oplus^r_{i=1} T_{\pi}$
has an open image. 

\noindent
{\rm (3)} $K_{\Phi}$ is a 
$\mbb{Z}_p^{\oplus r\cdot [k:\mbb{Q}_p]}$-extension of $Kk_{\pi}$.
\end{lemma}

To show this lemma, we use the following.

\begin{lemma}
\label{structure2-1}
The continuous cohomology group $H^1(Kk_{\pi}/K,T_{\pi})$ is finite.
\end{lemma}

\begin{proof}
Take a positive integer $m$ large enough.  We may suppose that 
the image of 
 $\chi_{\pi}\colon \mrm{Gal}(Kk_{\pi}/K)\to \cO_k^{\times}$ 
contains $1+\pi^m\cO_k$.
Let $K'$ be the subfield of $Kk_{\pi}/K$ such that 
$\chi_{\pi}(\mrm{Gal}(Kk_{\pi}/K'))=1+\pi^m\cO_k$.
In particular, $Kk_{\pi}/K'$ is a $\mbb{Z}_p^d$-extension where
$d:=[k:\mbb{Q}_p]$. 
Take a subfield $M$ in $Kk_{\pi}/K'$ with the property that 
$Kk_{\pi}/M$ is a $\mbb{Z}_p$-extension.
Let $\sigma_0$ be a topological generator of $\mrm{Gal}(Kk_{\pi}/M)$ and $c$ 
the  positive integer such that $\pi^{-c}(\chi_{\pi}(\sigma_0)-1)\in \cO_k^{\times}$.
Now we consider the following exact sequence
$$
0\to H^1(M/K',T_n^{\mrm{Gal}(Kk_{\pi}/M)})
\to H^1(Kk_{\pi}/K',T_n)\to H^1(Kk_{\pi}/M,T_n)
$$
where $T_n:=T_{\pi}/\pi^nT_{\pi}$. 
For any $n>c$, we have
\begin{itemize}
\item[--] $H^1(Kk_{\pi}/M,T_n)\simeq T_n/(\sigma_0-1)T_n
= T_n/\pi^cT_n$ (cf.\ \cite[\S 1]{Wa}), and 
\item[--] $T_n^{\mrm{Gal}(Kk_{\pi}/M)}=T_n^{\sigma_0=1}=\pi^{n-c}T_n$.
\end{itemize}
In particular,
$p^c$ vanishes 
$H^1(Kk_{\pi}/M,T_n)$ and 
$T_n^{\mrm{Gal}(Kk_{\pi}/M)}$ for any $n$.
Thus we see that 
$p^{2c}$ vanishes 
$p^{2c}H^1(Kk_{\pi}/K',T_n)$. 
Since we have an exact sequence
$$
0\to H^1(K'/K,T_n^{\mrm{Gal}(Kk_{\pi}/K')})
\to H^1(Kk_{\pi}/K,T_n)\to H^1(Kk_{\pi}/K',T_n),
$$
we obtain the  fact that 
$p^{3c}$ vanishes $H^1(Kk_{\pi}/K,T_n)$. 
Since we have an isomorphism
$H^1(Kk_{\pi}/K,T_{\pi})\simeq \plim_{n}H^1(Kk_{\pi}/K,T_n)$
(cf.\ \cite[Chapter II, Corollary 2.7.6]{NSW}),
we see that $H^1(Kk_{\pi}/K,T_{\pi})$ is killed by $p^{3c}$. 
Hence the proof finishes if we show that $H^1(Kk_{\pi}/K,T_{\pi})$
is a finitely generated $\mbb{Z}_p$-module.
We have an injection 
$H^1(Kk_{\pi}/K,T_{\pi})/pH^1(Kk_{\pi}/K,T_{\pi})
\hookrightarrow H^1(Kk_{\pi}/K,T_e)$
where $e$ is the ramification index of $k/\mbb{Q}_p$.
Since $\mrm{Gal}(Kk_{\pi}/K)$, isomorphic 
to an open subgroup of $\cO_k^{\times}$, 
is topologically finitely generated,
we see that the set of continuous 1-cocycles 
from $\mrm{Gal}(Kk_{\pi}/K)$ to $T_e$ is finite.
Hence $H^1(Kk_{\pi}/K,T_e)$ is finite,
and thus so is $H^1(Kk_{\pi}/K,T_{\pi})/pH^1(Kk_{\pi}/K,T_{\pi})$. 
By \cite[Chapter II, Corollary 2.7.9]{NSW},
we obtain the fact that 
$H^1(Kk_{\pi}/K,T_{\pi})$
is a finitely generated $\mbb{Z}_p$-module as desired.
\end{proof}

\begin{proof}[Proof of Lemma \ref{structure2-2}]
We follow the proofs  of Lemmas 2.12 and 2.13 of \cite{BGK}.

\noindent
(1) We define an $\cO_k$-linear map 
$\xi_n\colon F_{\phi}(\mbf{m}_K)/[\pi^n]_{\phi} F_{\phi}(\mbf{m}_K)
\to H^1(Kk_{\pi},F_{\phi}[\pi^n]_{\phi})$
to be the composite of the Kummer map
$F_{\phi}(\mbf{m}_K)/[\pi^n]_{\phi} F_{\phi}(\mbf{m}_K)
\hookrightarrow H^1(K,F_{\phi}[\pi^n]_{\phi})$
and the restriction map
$H^1(K,F_{\phi}[\pi^n]_{\phi})\to H^1(Kk_{\pi},F_{\phi}[\pi^n]_{\phi})$.
We have $\xi_n(a_i\ \mrm{mod}\ [\pi^n]_{\phi} F_{\phi}(\mbf{m}_K))=\vphi_i^{(n)}$.
Since $F_{\phi}(\mbf{m}_K)$ is $\pi$-adically complete, 
by taking the inverse limit of 
$(\xi_n)_n$, we obtain a morphism
$$
\xi\colon F_{\phi}(\mbf{m}_K)\hookrightarrow  
H^1(K, T_{\pi})\to H^1(Kk_{\pi}, T_{\pi})
$$
of $\cO_k$-modules
(here, we use \cite[Chapter II, Corollary 2.7.6]{NSW}).
By definition, we have  $\xi(a_i)=\vphi_i$ for any $i$.
By the equation $\lambda_1\varphi_1+\cdots + \lambda_r\varphi_r=0$,
we obtain 
$\xi([\lambda_1]_{\phi}(a_1)\oplus \cdots \oplus [\lambda_r]_{\phi}(a_r))=0$.
Since the kernel of the restriction map 
$H^1(K, T_{\pi})\to H^1(Kk_{\pi}, T_{\pi})$
is isomorphic to $H^1(Kk_{\pi}/K,T_{\pi})$ and this is finite 
by Lemma \ref{structure2-1},
we obtain that 
$[\lambda_1]_{\phi}(a_1)\oplus \cdots \oplus [\lambda_r]_{\phi}(a_r)$
is an torsion element of the $\cO_k$-module $F_{\phi}(\mbf{m}_K)$.
Since  $a_1,\dots ,a_r$ are linearly independent over $\cO_k$,
we obtain $\lambda_1=\cdots =\lambda_r=0$.

\noindent
(2) Put $V=T_{\pi}\otimes_{\mbb{Z}_p} \mbb{Q}_p = T_{\pi}\otimes_{\cO_k} k$
and $W=\oplus^r_{i=1} V$.
Then $W$ is a semi-simple $k[G_K]$-module and its $G_K$-action factors through 
$\mrm{Gal}(Kk_{\pi}/K)$. 
Set $M:=\mrm{Im}(\Phi) \otimes_{\mbb{Z}_p}\mbb{Q}_p$.
Since we have $\tau\Phi(\sigma)=\Phi(\tau \sigma \tau^{-1})$
for any $\tau\in G_K$ and $\sigma\in G_{Kk_{\pi}}$,
$\mrm{Im}(\Phi)$ is a
$\mbb{Z}_p[G_K]$-stable submodule of $\oplus^r_{i=1} T_{\pi}$.
(The author does not know whether 
the $\cO_k$-action on $\oplus^r_{i=1} T_{\pi}$ preserves 
$\mrm{Im}(\Phi)$ or not.)
In particular, 
$M$ is a $\mbb{Q}_p[G_K]$-stable submodule of $W$.

We claim that $M$ is a $k$-stable submodule of $W$. 
Put $G_{\pi}:=\chi_{\pi}(G_K)$. This is an open subgroup of 
$\mrm{GL}_{\cO_k}(T_{\pi})=\cO_k^{\times}$
and $\mbb{Z}_p[G_{\pi}]$ is a $\mbb{Z}_p$-subalgebra of 
$\mrm{End}_{\cO_k}(T_{\pi})=\cO_k$.
We see that $\mbb{Z}_p[G_{\pi}]$ is open in $\cO_k$
(in fact, if we take $m>0$ large enough so that 
$1+\pi^m\cO_k\subset G_{\pi}$,
then $\mbb{Z}_p[G_{\pi}]$ contains $\pi^m\cO_k$).
Thus we have  $\mbb{Q}_p[G_{\pi}]=\mrm{End}_{k}(V)=k$.
Since $\mbb{Z}_p[G_{\pi}]$-action on $\oplus^r_{i=1} T_{\pi}$ 
preserves $\mrm{Im}(\Phi)$,
the claim follows.

We show $M=W$. Assume that $M$ is strictly smaller than $W$.
Since $W$ is semi-simple as a $k[G_K]$-module,
there exists a $k[G_K]$-stable submodule $M_1$ of $W$ 
such that $W=M\oplus M_1$.
By the assumption, we know that $M_1$ is not zero. 
Take an integer $i$ such that the projection to the $i$-th component
$p_i\colon W\to V$ does not vanishes $M_1$. 
Let $\eta\colon W\to V$ be the composite 
of $0_M\oplus \mrm{id}_{M_1}\colon W=M\oplus M_1\to M\oplus M_1=W$
and $p_i$. 
By definition, $\eta$ is a morphism of $k[G_K]$-modules 
and we have $\eta|_M=0$ and $\eta|_{M_1}\not= 0$.
We denote by $\eta_j \colon V\to V$ the composite
of the injection $V\hookrightarrow W$ to the $j$-th component and 
$\eta$. We see 
$\eta(v)=\sum^r_{j=1} \eta_j(v_j)$
for any $v=(v_1,\dots , v_r)\in W$.
Since $\eta_j\in \mrm{End}_k(V)$, we may regard $\eta_j$ 
as an element of $k$. Then, we have 
$
\eta(v)=\sum^r_{j=1} \eta_j v_j
$
for any $v=(v_1,\dots , v_r)\in W$.
Since $\eta|_M=0$, we obtain
$$
\sum^r_{j=1} \eta_j(\vphi_j(\sigma)\otimes 1)=0
$$
for any $\sigma \in G_{Kk_{\pi}}$ (here, we consider 
$\vphi_j(\sigma)\otimes 1$ as an element of 
the tensor product $T_{\pi}\otimes_{\cO_k} k$).
Taking $N>0$ large enough so that $\pi^N\eta_j\in \cO_k$ for all $j$,
we have 
$$
\left(\sum^r_{j=1} \pi^N\eta_j \vphi_j(\sigma)\right)\otimes 1=0.
$$
Since the natural map $T_{\pi}\to V$ is injective,
we obtain $\sum^r_{j=1} \pi^N\eta_j \vphi_j(\sigma)=0$.
By (1), we have $\pi^N\eta_1=\cdots =\pi^N\eta_r=0$,
and thus $\eta_1=\cdots =\eta_r=0$. 
This shows $\eta=0$. 
This contradicts the fact that $M_1$ is not zero and 
$\eta|_{M_1}\not=0$.
Therefore, we obtain $M=W$.
Since both $\mrm{Im}(\Phi)$ and $\oplus^r_{i=1} T_{\pi}$ 
are $\mbb{Z}_p$-lattices in $W$, 
$\mrm{Im}(\Phi)$ is open in $\oplus^r_{i=1} T_{\pi}$.

\noindent
(3) Since we have $\mrm{Gal}(K_{\Phi}/Kk_{\pi})\simeq \mrm{Im}(\Phi)$,
the result follows from (2).
\end{proof}

\begin{proof}[Proof of  Theorem \ref{structure:Ktilde}]
The result is  an immediate consequence of Lemmas 
\ref{structure1} and \ref{structure2-2}.
\end{proof}

\begin{remark}
The author believe that, in Lemma \ref{structure2-2} (2), the image of 
the homomorphism
$\Phi \colon G_{Kk_{\pi}}\to \oplus^r_{i=1} T_{\pi}$ 
should be stable under the $\cO_k$-action of  $\oplus^r_{i=1} T_{\pi}$.
If this is true, the Galois group 
$\mrm{Gal}(\widetilde{K}/K)$ has a structure of 
$\cO_k$-modules, free of rank $[K:k]$. 
Moreover, Lemma \ref{Hlem} should hold also for any $\sigma$
(without the assumption that $\chi_{\pi}(\sigma)$ is a rational integer).
\end{remark}

\subsection{Proofs of Theorems \ref{MMT} and \ref{MT}}

We show Theorems \ref{MMT} and \ref{MT}
given in the Introduction.  

\begin{proof}[Proof of Theorem \ref{MMT}]
Assume that there exist a finite extension $L/\widetilde{K}$ 
and an abelian variety $A$ over $L$ with potential good reduction 
such that  $A(L)_{\mrm{tor}}$ is infinite.
The goal is to check that  $(k,\pi)$ satisfies (W) under this assumption. 
We remark that  (W) in Theorem \ref{MMT} 
coincides with (W)'  in Theorem \ref{MTab} for $k=E$ and $\chi=\chi_{\pi}$
(here, we recall that $\delta_{\chi_{\pi}}=\pi$; see Example \ref{exLT}).
Hence, by Theorem \ref{MTab}, 
it suffices to show  that there exists a finite extension $K'/K$ so that  
$A$ is defined over $K'$ and $A(K'k_{\pi})[p^{\infty}]$ is infinite.
At first, we take a finite extension $K_1/K$ so that 
$L\subset K_1\widetilde{K}$, $A$ is defined over $K_1$ and has good reduction over $K_1$.
By the same method as the construction of $\widetilde{K}$ from $K,k, \pi$ (and $\phi$),
we define $\widetilde{K}_1$ to be the filed corresponding to $K_1,k, \pi$ (and $\phi$).
Clearly we have $L\subset \widetilde{K}_1$ and thus 
$A(\widetilde{K}_1)_{\mrm{tor}}$ is infinite. 
Since the residue field of $\widetilde{K}_1$ is finite by Proposition \ref{finres}, 
it follows from \cite[Proposition 2.9]{Oz} that 
the prime-to-$p$ part of $A(\widetilde{K}_1)_{\mrm{tor}}$ is finite.
Thus  $A(\widetilde{K}_1)[p^{\infty}]$ is infinite.
If we denote by $V_p(A)$ the rational $p$-adic Tate module of $A$, 
then the infiniteness of  $A(\widetilde{K}_1)[p^{\infty}]$ implies that
$\widetilde{V}:=V_p(A)^{G_{\widetilde{K}_1}}$ is a non-zero $G_{K_1}$-stable submodule of $V_p(A)$.
We regard $\widetilde{V}$ as a representation of $\mrm{Gal}(\widetilde{K}_1/K_1)$.
By Proposition \ref{unipotent} (with $G:=\mrm{Gal}(\widetilde{K}_1/K_1)$ 
and $H:=\mrm{Gal}(\widetilde{K}_1/K_1k_{\pi})$), 
there exists a finite Galois extension 
$K'/K_1$ in $\widetilde{K}_1/K_1$ such that $H':=\mrm{Gal}(\widetilde{K}_1/K'k_{\pi})$ 
acts unipotently on $\widetilde{V}$.
Hence $\widetilde{V}^{H'}=V_p(A)^{G_{K'k_{\pi}}}$ is non-zero,
that is,  $A(K'k_{\pi})[p^{\infty}]$ is infinite as desired. 
\end{proof}

\begin{proof}[Proof of Theorem \ref{MT}]
Assume that neither ($\mu$) nor (W) in the Introduction holds. 
By Corollary \ref{keycor}, 
for the proof of the theorem, 
it suffices to check {\rm ($\mu_{\infty}$)} and {\rm (AV${}_{\infty}$)} for $(L=\tilde{K}/K,g=\infty)$. 
The condition {\rm (AV${}_{\infty}$)} 
is a consequence of Theorem \ref{MMT}. 
For  {\rm ($\mu_{\infty}$)},
it suffices to show that the set $\mu_{\infty}(L')$ is finite for any finite extension  $L'/L$. 
Since the residue field of $L'$ is finite by Proposition \ref{finres},
the finiteness of the set $\mu_{\ell^{\infty}}(L')$ for any prime number $\ell\not=p$ follows. 
Furthermore, we see that ${L'}$ does not 
contain $\mu_{\ell}$ 
if $\ell\ge q_{L'}$ where 
$q_{L'}$ is the order of the residue field $\mbb{F}_{L'}$ of ${L'}$.
In fact, if $L'$ contains $\mu_{\ell}$,
then $\mbb{F}_{L'}$ contains the residue field of $\mbb{Q}_p(\mu_{\ell})$.
Hence, taking $f_{\ell}>0$ the minimum integer $s$ such that $p^s\equiv 1\ \mrm{mod}\ \ell$,
we have $q_{L'}\ge p^{f_{\ell}}\ge \ell+1$. 
Finally, we show the finiteness of the set $\mu_{p^{\infty}}(L')$.
Assume that  $\mu_{p^{\infty}}(L')$ is infinite. Then $L'$ contains $k(\mu_{p^{\infty}})$.
By \cite[Lemma 2.7]{Oz} and the assumption that $(\mu)$ does not hold,  
we know that $k_{\pi}\cap k(\mu_{p^{\infty}})$ is a finite extension of $k$. 
Thus the extension $k_{\pi}(\mu_{p^{\infty}})/k_{\pi}$ is of infinite degree.
It follows from local class field theory that 
the residue field of $k_{\pi}(\mu_{p^{\infty}})$ is infinite.
Since $L'$ contains $k_{\pi}(\mu_{p^{\infty}})$, we obtain the fact that 
the residue field of $L'$ is also infinite but this is a contradiction.
Therefore, the set $\mu_{p^{\infty}}(L')$ must be finite as desired.
\end{proof}


\subsection{Proof of Theorem \ref{MTF}}
\label{pfMTF}

We show Theorem \ref{MTF}.
Assume that we have found
a finite set $\mcal{W}_{\mrm{ab}}=\mcal{W}_{\mrm{ab}}(f,g;k)$
of $q$-Weil integers with the property described in the theorem 
under the additional condition that $G$ is an abelian variety 
with potential good reduction.
Assuming this, we show that the set
$$
\mcal{W}=\mcal{W}(f,g;k):=\mcal{W}_{\mrm{ab}}\cup q\cdot \mu_{p-1}
$$
satisfies the desired property. 
Assume that we have  $\mrm{Nr}_{k/\mbb{Q}_p}(\pi)\notin \mcal{W}$.
Let $K$ be a finite extension of $k$ with $f_K\le f$. 
Take a finite extension $L/\widetilde{K}$.
Replacing $L$ by a finite extension, we may assume that $L/K$ is Galois. 
We consider 
($\mu_{\infty}$) and (AV${}_{\infty}$) for  $(L/K,g)$.
The condition ($\mu_{\infty}$) follows from the facts that 
$q^{-1}\mrm{Nr}_{k/\mbb{Q}_p}(\pi)$ is now not a root of unity and 
any finite extension of $\widetilde{K}$ has a finite residue field
(see the proof of Theorem \ref{MT}). 
Furthermore, 
the condition $\mrm{Nr}_{k/\mbb{Q}_p}(\pi)\notin \mcal{W}_{\mrm{ab}}$
assures (AV${}_{\infty}$).
Now the theorem follows from  Proposition \ref{key} (2).

In the rest of the proof, 
we show the existence of $\mcal{W}_{\mrm{ab}}=\mcal{W}_{\mrm{ab}}(f,g;k)$. 
First we consider the case where $k$ is a Galois extension of $\mbb{Q}_p$.
Let $K$ be a finite extension of $k$ with $f_K\le f$
and $A$ an abelian variety over $K$, of dimension at most $g$,
with potential good reduction. 
Suppose that $A(L)_{\mrm{tor}}$ is infinite 
for a finite extension $L/\widetilde{K}$.
Replacing $K$ by a totally ramified extension,
we may suppose that $A$ has good reduction over $K$ 
(cf. \cite[\S 2]{ST}).
The argument given in the proof of Theorem \ref{MMT} shows that 
there exists a finite extension $K'/K$  with the property that 
$V_p(A)^{G_{K'k_{\pi}}}$ is not zero. 
Let $M$ be the Galois closure of $K'k_{\pi}$ over $K$,
which is a finite extension of $Kk_{\pi}$. 
Then $(V_p(A)^{G_M})^{\vee}$ 
is a non-zero crystalline representation of $G_K$
with Hodge-Tate weights in $\{-1,0\}$.
Here, the notation ``$\vee$'' stands for the usual dual representation.
By \cite[Lemma 2.5]{Oz} (or Remark \ref{remark}),
there exist finite extensions $K'/K$ and $E/\mbb{Q}_p$ with $K',E\supset k$   
such that any Jordan-H\"older factor $W$ of 
the $E$-representation 
$((V_p(A)^{G_M})^{\vee}\otimes_{\mbb{Q}_p} E)|_{G_{K'}}$
is of the form 
\begin{equation}
\label{charLT}
W=E(\hat{\chi}^{-1}),\quad
\hat{\chi}=\prod_{\sigma \in \Gamma_k}\sigma^{-1}\circ \chi^{-r_{\sigma}}_{\pi}
\end{equation}
for some integer $r_{\sigma}\in \{-1,0\}$. 
We fix  a lift $\hat \sigma\colon \overline{\mbb{Q}}_p\to \overline{\mbb{Q}}_p$
of each $\sigma\in \Gamma_k$
and denote by $\Gamma$ the set of homomorphisms $\hat \sigma$ such that $r_{\sigma}=-1$.
We have $\hat{\chi}=\prod_{\hat{\sigma} \in \Gamma}\hat{\sigma}^{-1}\circ \chi_{\pi}$
and $|\Gamma|=\sum_{\sigma\in \Gamma_k}(-r_{\sigma})$.
Let  $\alpha$ be a root of the characteristic polynomial of $D_{\mrm{cris}}^{K'}(W)$.
It follows from  Lemma \ref{root} 
that  $\alpha$ is of the form
\begin{equation}
\label{alpha}
\alpha=a^{f_{K'/k}},\quad a= \prod_{\tau\in \Gamma_k} \tau(\pi)^{t_{\tau}}
\end{equation}
for some integers $t_{\tau}\ge 0$  such that 
$\sum_{\tau\in \Gamma_k} t_{\tau}
=\sum_{\hat{\sigma} \in \Gamma } 1 
=\sum_{\sigma\in \Gamma_k} (-r_{\sigma})$.
Since $W$ is a subquotient $E$-representation of 
$(V_p(A)^{G_M})^{\vee}\otimes_{\mbb{Q}_p} E$
and 
$(V_p(A)^{G_M})^{\vee}$ is a quotient $\mbb{Q}_p$-representation of 
$V_p(A)^{\vee}$, 
we know that $\alpha$ is a root of the characteristic polynomial of 
$D_{\mrm{cris}}^{K'}(V_p(A)^{\vee})$. 
Since $A$ has good reduction over $K$, 
$\alpha$ is a $f_{K'/K}$-th power of 
some root $\beta$ of the characteristic polynomial of 
$D_{\mrm{cris}}^{K}(V_p(A)^{\vee})$;  
\begin{equation}
\label{beta}
\alpha=\beta^{f_{K'/K}}.
\end{equation}
By the Weil conjecture, 
$\beta$ is a $q_K$-Weil integer of weight $1$.
Set $t_0:=\sum_{\tau\in \Gamma_k} t_{\tau}
=\sum_{\sigma\in \Gamma_k} (-r_{\sigma})$.
We have $0\le t_0\le [k:\mbb{Q}_p]$ since $r_{\sigma}\in \{-1,0\}$.
Furthermore, we have $t_0\not=0$.
In fact, if we assume  $t_0=0$,
then $t_{\tau}=0$ for any $\tau$
and thus $\beta^{f_{K'/K}}=1$. However, this contradicts the fact that 
$\beta$ is a $q_K$-Weil integer of non-zero weight $1$.
By \eqref{alpha} and  \eqref{beta},
we have $a^{f_K}=\zeta \beta^{f_k}$ where $\zeta$ is a root of unity.
Let $T_1=T_1(f,g)$ be the set of $p$-Weil integers $x$ of weight at most $f$
such that $[\mbb{Q}(x):\mbb{Q}]\le 2g$.  
The set $T_1$ is finite.
Let $k_1=k_1(f,g;k)$ be the extension field of $k$ 
obtained by adjoining $T_1$.
Then $k_1$ depend only on $f,g$ and $k$.
By the Weil conjecture, we have $\beta\in T_1$.
This in particular implies $\beta, \zeta\in k_1$.
Taking $\mrm{Nr}_{k_1/\mbb{Q}_p}$ to the equality
$a^{f_K}=\zeta \beta^{f_k}$,
we obtain that 
$$
\mrm{Nr}_{k/\mbb{Q}_p}(\pi)^{t_0f_K[k_1:k]}
\mrm{Nr}_{k_1/\mbb{Q}_p}(\beta)^{-f_k}
$$
is a root of unity in $\mbb{Q}_p$, that is,
a $(p-1)$-th root of unity. 
Since we have $0< t_0\le [k:\mbb{Q}_p]$, 
the existence of the desired set $\mcal{W}_{\mrm{ab}}=\mcal{W}_{\mrm{ab}}(f,g;k)$ now follows.

Next we consider the case where $k$ is not necessarily 
a Galois extension of $\mbb{Q}_p$. 
Set $f_{(k)}:=f\cdot e_{k_G/k}\cdot [k_G:k]$.
For any finite Galois extension $F/\mbb{Q}_p$, 
we already knows the existence of $\mcal{W}_{\mrm{ab}}(f_{(k)},g;F)$.
We fix the choice of $\mcal{W}_{\mrm{ab}}(f_{(k)},g;F)$
for each such $F$.
Let $\mrm{Gal}_k$ be  the set of Galois extensions $F/\mbb{Q}_p$
such that $F\supset k_G$ and $[F:k_G]\le e_{k_G/k}$.
(If $k/\mbb{Q}_p$ is Galois, then we have $f_{(k)}=f$ and $\mrm{Gal}_k=\{ k \}$.)
We define 
$$
\mcal{W}_{\mrm{ab}}=\mcal{W}_{\mrm{ab}}(f,g;k):=\{ x\in \overline{\mbb{Q}} \mid 
x^{f_{F/k}}\in \mcal{W}_{\mrm{ab}}(f_{(k)},g;F) \ \mbox{for some $F\in \mrm{Gal}_k$} \}.
$$ 
Since $\mrm{Gal}_k$ is finite, the set $\mcal{W}_{\mrm{ab}}$ is also finite.
It suffices to show that $\mcal{W}_{\mrm{ab}}$ satisfies the desired property.
Let $K$ be a finite extension of $k$ with $f_K\le f$
and $A$ an abelian variety over $K$, of dimension at most $g$,
with potential good reduction. 
Suppose that $A(L)_{\mrm{tor}}$ is infinite 
for a finite extension $L/\widetilde{K}$.
As explained above, we may assume that $A$ has good reduction over $K$, 
and there exists  a finite extension of $M/Kk_{\pi}$
such that $M$ is a Galois extension of $K$ and  $V_p(A)^{G_M}$ is not zero.
The torsion subgroup $A(M)_{\mrm{tor}}$ of $A(M)$ is now infinite. 
On the other hand, 
it follows from \cite[Lemma 2.8]{Oz} that
there exist $k'\in \mrm{Gal}_k$ and a uniformizer $\pi'$ of $k'$ 
with the properties that 
$\mrm{Nr}_{k'/k}(\pi')=\pi^{f_{k'/k}}$ and
$k_{\pi}\subset k'_{\pi'}$.
Putting $K'=Kk'$, we have 
$f_{K'}\le f_K\cdot [K':K]\le f_K\cdot [k':k]\le f_{(k)}$.
We denote by $\widetilde{K}'$ the field corresponding to $K',k'$ and $\pi'$
(cf.\ Proposition \ref{tildeK} (2)).
Putting $L'=M\widetilde{K}'$, then $L'$ is a finite extension of $\widetilde{K}'$
and we have  $A(L')_{\mrm{tor}}$ is infinite.
Therefore,
we have 
$\mrm{Nr}_{k'/\mbb{Q}_p}(\pi')\in \mcal{W}_{\mrm{ab}}(f_{(k)},g;k')$.
Since we have  $k'\in \mrm{Gal}_k$ and 
$\mrm{Nr}_{k'/\mbb{Q}_p}(\pi')=\mrm{Nr}_{k/\mbb{Q}_p}(\pi)^{f_{k'/k}}$,
we obtain $\mrm{Nr}_{k/\mbb{Q}_p}(\pi)\in\mcal{W}_{\mrm{ab}}(f,g;k)$.
This finishes the proof.

\end{document}